\documentclass{amsart}
\usepackage{amsmath,amsthm}     
\usepackage{amssymb,bbm}            
\usepackage{euscript}           
\usepackage{graphicx}          
\usepackage{array,enumerate,calc,graphicx,xcolor,tikz,tikz-cd, enumitem,mathtools}
\usetikzlibrary{matrix,arrows,decorations.pathmorphing}
\usepackage{cancel}
\usepackage[colorlinks]{hyperref}
\hypersetup{citecolor=blue}
\usepackage{multirow}

\newtheorem{thm}[subsection]{Theorem}
\newtheorem{prop}[subsection]{Proposition}

\newtheorem{cor}[subsection]{Corollary}
\newtheorem{lemma}[subsection]{Lemma}
\newtheorem*{introthm}{Theorem}
\newtheorem*{introcor}{Corollary}

\theoremstyle{definition}  
\newtheorem{defn}[subsection]{Definition}
\newtheorem{example}[subsection]{Example}

\newtheorem{remark}[subsection]{Remark}
\newtheorem{notation}[subsection]{Notation}

\numberwithin{equation}{section}

\newcommand{\N}{\mathbb{N}} 
\DeclareMathOperator{\Res}{Res}
\newcommand{\rst}[3]{\Res_{#2}^{#1}{#3}}
\newcommand{\op}{\operatorname{op}}
\DeclareMathOperator{\Tr}{Tr} 
\DeclareMathOperator{\Cat}{Cat} 
\DeclareMathOperator{\id}{id} 
\DeclareMathOperator{\wfs}{WFS} 

\newlength\squareheight
\newcommand{\boxslash}{\tikz{\draw (0,0) rectangle (\squareheight,\squareheight); \draw[-] (0,0) -- (\squareheight,\squareheight)}}

\newcommand{\T}{\mathcal{T}}
\newcommand{\A}{\mathcal{A}}
\newcommand{\B}{B}
\renewcommand{\P}{\mathcal{P}}

\newcommand{\Rop}{{R}^{\op}}

\newcommand{\lift}{\boxslash}
\newcommand{\llR}{\prescript{\lift}{}{R}}
\newcommand{\lRl}{(^{\lift} R)^{\lift}}
\newcommand{\meet}{\wedge}
\newcommand{\join}{\vee}

\newcommand{\vdot}{\begin{tikzpicture}\filldraw (0,0) circle (1pt);\end{tikzpicture}} 

\title{Transfer systems on trees and opposite posets}
\author{Andrew Fargo}
\author{Christy Hazel}
\author{Deven Platt} 

\address{Department of Mathematics\\Grinnell College\\ Grinnell, IA 50112}
\email{fargodre@grinnell.edu}
\address{Department of Mathematics\\Grinnell College\\ Grinnell, IA 50112}
\email{hazelchristy@grinnell.edu}

\address{Department of Mathematics\\University of Colorado--Boulder\\ Boulder, CO 80309}
\email{Deven.Platt@colorado.edu}

\begin{document}
\begin{abstract}
   We study transfer systems on partially ordered sets whose Hasse diagrams are rooted trees. We prove there is a bijection between transfer systems on a tree and order-preserving functions to the natural numbers that are bounded by the rank function. We then develop a recurrence for the set of functions and use this to enumerate transfer systems on some families of trees.
   In addition to our study of trees, we prove for any finite poset $\P$, the lattice of transfer systems on the opposite poset $\P^{\op}$ is isomorphic to the opposite lattice of transfer systems on $\P$.
   We show this by establishing a bijection between weak factorization systems and transfer systems, which builds on the results for lattices shown in previous work of Franchere, Ormsby, Osorno, Qin, and Waugh.
\end{abstract}

\maketitle
\tableofcontents

\section{Introduction}
 Transfer systems are important combinatorial objects in homotopy theory given their connections to equivariant operads and model structures. Roughly speaking, a \emph{transfer system} on a partially ordered set (poset) is a refinement of its relation that is reflexive, transitive, and satisfies an additional ``restriction condition'' (see Definition~\ref{def:transf}). Transfer systems were introduced in equivariant homotopy theory to study the various levels of commutativity that exist for equivariant ring spectra. The key connection is that the homotopy category of $G$-$N_\infty$-operads defined by Blumberg and Hill \cite{BH1} is equivalent to the category of transfer systems on the subgroup lattice of $G$ (see \cite{Rubinoperads, BBR1}). Motivated by a desire to understand $G$-$N_\infty$-operads for various finite groups $G$, there has been a surge of recent work exploring the structure of transfer systems on subgroup lattices (see \cite{Rubinoperads, steiner, BBR1, HMOO, BMO2, ranktwo, BMO, nonabelian} for some examples). 
 
 Since their introduction, transfer systems have been used to study model structures and weak factorizations systems in category theory---see \cite{selfdual, modelstr, modelstr2, bousfield1, bousfield2} for some examples where connections between transfer systems and categorical structures are explored and utilized. Thus there is motivation to understand the structure of transfer systems on general posets, not just on subgroup lattices.

In this paper we start by investigating transfer systems on graded posets whose Hasse diagrams are rooted trees.  We call such graded posets \emph{arborescences} (see Definition~\ref{def:arbor}). We develop a recurrence relation in Section~\ref{sec:recur} that we use to enumerate transfer systems on arborescences in a handful of examples. In previous work of Balchin, Barnes, and Roitzheim \cite{BBR1} they studied transfer systems on totally order chains. Any tree can be obtained by gluing together chains of various lengths, and thus using pasting techniques to explore trees is a natural next step. We also provide a bijection between transfer systems on bifurcating trees and certain types of ``coalescent histories'' from computational biology (see Appendix~\ref{ap:coalhis}). This connection further inspired our interest in transfer systems on trees.

In the second part of this paper, we prove there is an order-reversing bijection between transfer systems on a poset $\P$ and transfer systems on its opposite poset $\P^{\op}$ (see Section~\ref{sec:op}). We obtain this by establishing a bijection between transfer systems and weak factorization systems from category theory.  The results we prove in this section hold for any finite poset, not just for arborescences. This builds on a similar result for weak factorization systems on lattices proven in \cite{selfdual}. \smallskip

To give an introduction to our results, suppose we have an arborescence $\A$, such as the one in Figure~\ref{fig:introex}. Our convention is that the elements of $\A$ are ordered so that the root in the corresponding Hasse diagram is the unique minimal element. Observe there is a rank function $\rho\colon \A\to \N$ that sends a vertex to the length of the unique path from the root to that vertex. For example, in Figure~\ref{fig:introex} the subscript of each vertex is equal to its rank.
\begin{figure}[ht]
\begin{center}\scalebox{0.9}{
    \begin{tikzpicture}
        \node (0) at (0,0) {$\bullet$} node[below]{$x_0$};
        \node (10) at (-1,1) {$\bullet$};
        \draw (10) node[left]{$x_1$};
        \node (11) at (0,1) {$\bullet$};
        \draw (11)  node[above]{$y_1$};
        \node (12) at (1,1) {$\bullet$};
        \draw (12)  node[right]{$z_1$}; 
        \node (20) at (-1.5,2) {$\bullet$};
        \draw (20) 
        node[above]{$x_2$};
        \node (21) at (-0.5,2) {$\bullet$};
        \draw (21) node[above]{$y_2$};
        \node (22) at (1.5,2) {$\bullet$};
        \draw (22)  node[above]{$z_2$};
        \draw[thick] (0.center)--(10.center)--(20.center);
        \draw[thick] (0.center)--(12.center)--(22.center);
        \draw[thick] (0.center)--(11.center);
        \draw[thick] (10.center)--(21.center);
    \end{tikzpicture}}
    \caption{A Hasse diagram for an arborescence.}
    \label{fig:introex}
\end{center}
\end{figure}
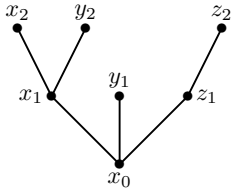
Our first result is that transfer systems on an arborescence are in bijection with order preserving functions that are bounded by the rank function.  In what follows we will write $\Tr(\P)$ for the lattice of transfer systems on a poset $\P$.

\begin{introthm}[Theorem~\ref{thm:arbij}]
Let \(\mathcal{A}\) be an arborescence, and let \(M_\mathcal{A}\) be the set of functions
\[ \{\lambda\colon \mathcal{A}\to \N~|~ \text{\(\lambda\) is order-preserving and \(\lambda \leq \rho\)}\}.\]
Then there is an order-reversing bijection from $\Tr(\A)$ to $M_\A$. (Here $M_\A$ is a poset using the ordering $f\leq g$ if and only if $f(x)\leq g(x)$ for all $x\in \A$.)
\end{introthm}
We note this map does not give a poset isomorphism from $\Tr(\A)$ to $M_\A^{\op}$ in general. There can be relations in $M_\A$ that are not present in $\Tr(\A)$---see Figure~\ref{fig:2delta} in Section~\ref{sec:bijmono} for an example (for totally ordered chains, this is exactly the Stanley versus Tamari lattice structures on Catalan objects as discussed in Remark~\ref{rem:stanlat}). While we don't have the full strength of a lattice isomorphism, we are still able to use the result to produce enumeration results. For example, this provides an alternative proof of Balchin, Barnes, and Roitzheim's result \cite{BBR1} that the number of transfer systems on the chain $[n]=\{0<1<2<\dots<n\}$ is given by the Catalan number $\Cat(n+1)$ (see Corollary~\ref{cor:catalan}). 

In Section~\ref{sec:recur} we prove there is a recurrence relation for the number of functions in $M_\A$ and leverage the bijection to enumerate transfer systems in some new examples of trees. We introduce some notation here to explain the recurrence. For $x\in \A$, let $x^\uparrow=\{y\in \A ~|~ x\leq y\}$. Define 
\[\B(x,n)=\{ \sigma\colon x^{\uparrow}\to \N ~|~  \sigma \text{ is order-preserving}, ~\sigma \leq \rho|_{x^{\uparrow}}, \text{ and } \sigma(x) \geq n\}.\] We write $b(x,n)$ for the cardinality of $\B(x,n)$. By Theorem~\ref{thm:arbij} above we have:
\begin{introcor} Let $\A$ be an arborescence and $\bot$ be the minimal element in $\A$. Then 
$
\# \Tr(\A)=b(\bot, 0).
$
\end{introcor}
On the other extreme, observe if $y$ is a leaf then
\[
b(y,n)=\begin{cases} \rho(y)-n+1, \quad& \text{ if } n \leq \rho(y),\\
0, \quad & \text{ if } n> \rho(y).
\end{cases}
\]
We prove the following recurrence, which moves up the tree and terminates when we reach a leaf:
\begin{introthm}[Theorem~\ref{thm:recur}]\label{thm:recurintro}
    Let $\A$ be an arborescence and $x\in \A$. For $y\in \A$, we write $y\gtrdot x$ for ``$y$ covers $x$'' (that is, $x<y$ and there does not exist $z$ such that $x<z<y$). Then we have
    \[ b(x,n) = \sum_{i = n}^{\rho(x)} \prod_{y \gtrdot x} b(y,i).\]
\end{introthm}
To see the recurrence in action, let $\A$ be the arborescence whose Hasse diagram is shown in Figure~\ref{fig:introex}. Then we can quickly compute
\begin{align*}
\# \Tr(\A)&=b(x_0,0)\\
&=b(x_1,0)\cdot b(y_1,0)\cdot b(z_1,0)\\
&=\left(b(x_2,0)b(y_2,0)+b(x_2,1)b(y_2,1) \right)\cdot b(y_1,0)\cdot \left(b(z_2,0)+b(z_2,1) \right)\\
&=(3^2+2^2)\cdot 2\cdot (3+2)\\
&=130.
\end{align*}

We apply the recurrence in families of examples in Section~\ref{sec:recur}. Example~\ref{ex:catalan} uses the recurrence to provide yet another proof that the number of transfer systems on the totally ordered chain is given by the Catalan numbers. Example~\ref{ex:longfork} builds on this to enumerate transfer systems on the ``long fork of depth $n$'', which we denote by $\mathcal{F}_n$ (this is a rooted tree with two leaves and $n$ non-leaf vertices that form a chain). We are able to use the recurrence to derive a closed form
\[
\# \Tr(\mathcal{F}_n)=\frac{2(5n+3)}{(n+2)(n+3)}\binom{2n+1}{n}.
\]
In Example~\ref{ex:depth2} we provide a closed form for the number of transfer systems on any tree of depth $2$. Finally in Example~\ref{ex:binary} we give a recursive method to produce the sequence for the number of transfer systems on full binary trees of depth $n$. \smallskip

Our convention is to partially order the vertex set of a rooted tree so that the root is the minimal element. A natural question is what happens to the transfer system lattice if we instead choose the reverse ordering so the root is maximal. Section~\ref{sec:op} explores this question in a general setting: what is the relationship between the lattice of transfer systems on a poset and the lattice of transfer systems on the ``opposite poset''? Given a poset $\P$ we write $\P^{\op}$ for the poset with the opposite ordering, i.e. $x\leq y$ in $\P^{\op}$ if and only if $y\leq x$ in $\P$. In \cite{selfdual} the authors show that if $\P$ is a self-dual lattice (i.e., if $\P\cong \P^{\op}$) then $\Tr(\P)$ is also self-dual. To prove this they establish there is a bijection between transfer systems and weak factorization systems for lattices (see Section~\ref{sec:op} for the definition of weak factorization system). Inspired by their techniques, we prove the following generalizations:
\begin{introthm}[Theorem~\ref{thm:wfs}]
    Let $\P$ be any finite poset. Then the function $\Tr(\P)\to \wfs(\P)$ given by $R\mapsto (\prescript{\lift}{}{R}, R)$ is an isomorphism of posets.
\end{introthm}
\begin{introcor}[Corollary~\ref{cor:oppable}]\label{thm:oppableintro}
    Let $\P$ be any finite poset. Then there is an order reversing bijection between $\Tr(\P)$ and $\Tr(\P^{\op})$. In particular we have that $\Tr(\P)^{\op}$ is isomorphic to $\Tr(\P^{\op})$ as posets.
\end{introcor}
\noindent Figure~\ref{fig:yop} at the end of Section~\ref{sec:op} illustrates an example of the two lattices of transfer systems on $\P$ and  $\P^{\op}$.
\begin{remark}
    For transfer systems on posets that are not semi-meet lattices (that is, posets in which unique greatest lower bounds do not always exist), there are two different definitions for transfer systems in the literature. The difference lies in the ``restriction condition'': one version forces relations involving meets of elements, while the other version forces relations involving a generalization of meets, which we call pseudo-meets (see Definition~\ref{def:pseudomeet}). These definitions are not equivalent, and the above theorem only holds in generality when taking the pseudo-meet definition. See Remark~\ref{rem:difres} for a more in-depth discussion on the differences in these definitions. 
\end{remark}

\subsection*{Organization of paper} Section~\ref{sec:back} gives background on posets, arborescences, and transfer systems. In Section~\ref{sec:ideals} we introduce the notion of an ideal of a poset and prove a pasting theorem for transfer systems (Theorem~\ref{thm:pasting}). In Section~\ref{sec:bijmono} we first prove the bijection with order-preserving functions for totally ordered chains and then use the pasting theorem to generalize to arborescences. We then prove the recurrence in Section~\ref{sec:recur} and apply it to enumerate transfer systems in examples. In Section~\ref{sec:op} we generalize beyond arborescences and prove $\Tr(\P)^{\op}\cong \Tr(\P^{\op})$ for any finite poset $\P$. Finally, in Appendix~\ref{ap:coalhis} we discuss coalescent histories and prove the bijection between transfer systems on arborescences and coalescent histories. 

\subsection*{Acknowledgments} The work in this paper began as a summer research project completed at Grinnell College. We thank the Henry Luce Foundation for providing support to the second author, which allowed us to have this summer research experience. We thank Tarush Mayani for his involvement in the early stages of this research. We also acknowledge \cite{ninfty}, which was an invaluable resource for example generation as we started working. Many thanks also to Kyle Ormsby for providing helpful feedback on an earlier draft and for his suggestions regarding weak factorization systems. 

\section{Background on arborescences and transfer systems}\label{sec:back}
In this section we introduce our main objects of study, arborescences and transfer systems, and prove some preliminary results that will be helpful in the main sections of the paper. We first provide some quick background information on partially ordered sets. For a more thorough discussion, see the textbook \cite[Chapter 3]{stanley1}, for example. 

A set $\P$ with a relation $\leq$ is a \textbf{partially ordered set} (or \textbf{poset}) if $\leq$ is reflexive, transitive, and antisymmetric. Given a finite poset $\P$, its \textbf{Hasse diagram} is a directed graph whose vertices are the elements of $\P$ and whose edges are given by the \textbf{covering relations} (so we draw a directed edge from $x$ to $y$ if and only if $x<y$ and there exists no element $z\in P$ such that $x<z<y$). Note the Hasse diagram is enough to recover the poset $\P$ since we know a partial order must be reflexive and transitive. As a convention we usually draw Hasse diagrams so that edges are directed upward, and in that case we do not include the orientation of the edges. 

The \textbf{meet} of two elements $x,y\in \P$ is an element $x\meet y\in \P$ with the property that $x\meet y\leq x$, $x\meet y \leq y$, and for all $z\in \P$, if $z\leq x$ and $z\leq y$ then $z\leq x\meet y$ (that is, $x\meet y$ is the greatest lower bound of $x$ and $y$). The meet is unique provided it exists. The \textbf{join} of $x$ and $y$ is the least upper bound and is denoted $x\join y$. A \textbf{lattice} is a partially ordered set with the property that any two elements have a meet and a join. A \textbf{meet semi-lattice} is a partially ordered set in which any two elements have a meet (but not necessarily a join). Observe any finite meet semi-lattice necessarily has a unique minimal element that is comparable to all elements. A finite lattice has a unique minimal element and a unique maximal element. We will write $\bot$, $\top$ for these respective elements. 

In Section~\ref{sec:op} we work with posets that are not necessarily meet semi-lattices. Such posets sometimes contain ``pseudo-meets'', which will be important in our study of transfer systems. These are maximal elements that are less than or equal to both $x$ and $y$, but might not be comparable to \emph{all} elements less than or equal to $x$ and $y$, as is required of meets. Precisely, we have:

\begin{defn}\label{def:pseudomeet}
Let $\P$ be a poset with $x,y\in \P$. We call an element $m\in \P$ a \textbf{pseudo-meet} of $x$ and $y$ if $m\leq x$, $m\leq y$, and for all $z\in \P$, if $z\leq x$, $z\leq y$, \emph{and} $z$ is comparable to $m$, then $z\leq m$.
\end{defn}

For example, $m_1$ and $m_2$ are two different pseudo-meets of $x$ and $y$ in the poset whose Hasse diagram is shown below. Note the meet of $x$ and $y$ does not exist because $m_1$ and $m_2$ are incomparable.
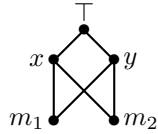
\begin{center}
    \begin{tikzpicture}[scale=0.8]
    \draw[thick] (0,-1)--(0,0)--(0.5,0.5)--(1,0)--(1,-1)--(0,0);
    \draw[thick] (0,-1)--(1,0);
    \draw (0,0) node{$\bullet$} node[left]{$x$};
    \draw (0,-1) node{$\bullet$} node[left]{$m_1$};
    \draw (1,-1) node{$\bullet$} node[right]{$m_2$};
    \draw (1,0) node{$\bullet$} node[right]{$y$};
    \draw (0.5,0.5) node{$\bullet$} node[above]{$\top$};
    \end{tikzpicture}
    \end{center}
One can give an analogous definition for \textbf{pseudo-joins}.\smallskip

In Sections~\ref{sec:bijmono} and \ref{sec:recur} we focus on posets whose Hasse diagrams are rooted trees where the minimal element is the root. The following conditions guarantee a poset's Hasse diagram is such a rooted tree: 

\begin{prop}
Let $\P$ be a finite poset with a unique minimal element $\bot$. The Hasse diagram for $\P$ is a rooted tree with root $\bot$ if and only if for all $x,y,z\in \P$, if $x\leq z$ and $y\leq z$, then $x\leq y$ or $y\leq x$. 
\end{prop}
\begin{proof}
First suppose the Hasse diagram is a rooted tree. Then $x\leq z$ and $y\leq z$ implies $z$ is a descendant of both $x$ and $y$. In a tree this occurs only if $x$ is an ancestor of $y$ or $y$ is an ancestor of $x$.
    
Now suppose the Hasse diagram is not a rooted tree. Then there must be a cycle in the graph of covering relations. Such a cycle must have at least $4$ vertices because a cycle with only $3$ vertices would be showing a transitive relation, which is not a covering a relation. This implies there exist distinct elements $w, x,y,z\in \P$ such that $w<x<z$ and $w<y<z$ but $x$ and $y$ are incomparable. See the diagram below:
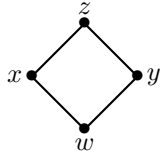
\begin{center}
    \begin{tikzpicture}[scale=0.7]
    \draw[thick] (0,0)--(1,1)--(0,2);
    \draw[thick] (0,0)--(-1,1)--(0,2);
    \draw (0,0) node{$\bullet$} node[below]{$w$};
    \draw (-1,1) node{$\bullet$} node[left]{$x$};
    \draw (1,1) node{$\bullet$} node[right]{$y$};
    \draw (0,2) node{$\bullet$} node[above]{$z$};
    \end{tikzpicture}
\end{center}
(note there are possibly more covering relations between any pair of related nodes in this figure; we are just drawing the Hasse diagram for the subposet $\{x,y,z,w\}$).
\end{proof}

Suppose $\A$ is a poset whose Hasse diagram is a rooted tree. Then $\A$ has extra structure in that $\A$ is a \textbf{graded poset}, meaning it comes equipped with function $\rho\colon \A\to \N$ with the property $\rho(y)=\rho(x)+1$ whenever $y$ covers $x$. We call such a function a \textbf{rank function}. When we are considering a tree as a graded poset with a specific rank function, we will call it an arborescence:

\begin{defn}\label{def:arbor}
A finite graded poset $\A$ is an \textbf{arborescence} if its Hasse diagram forms a rooted tree, and the rank function $\rho\colon \A\to \N$ is given by sending $x$ to the length of (i.e. number of edges in) the unique path in the Hasse diagram from the minimal element $\bot$ to vertex $x$.
\end{defn}

\begin{remark}\label{rem:treetoarb}
If we are given a finite graph that is a rooted tree, then we can construct an arborescence whose elements are the vertices of the graph and whose covering relations are given by the edges, which we take as directed to move away from the root. Note the Hasse diagram of the corresponding arborescence will exactly be the rooted tree. 
\end{remark}

We now define transfer systems and prove some nice properties about transfer systems on arborescences.

\begin{defn}\label{def:transf}
Suppose $(\P,\leq)$ is a finite poset. A \textbf{transfer system} on $\P$ is a reflexive and transitive relation $R$ that satisfies
    \begin{enumerate}[leftmargin=*]
        \item[(1)] (\emph{refinement of $\leq$}) for all $x,y\in \P$, if $x~R~y$ then $x\leq y$; and
        \item[(2)] (\emph{restriction condition}) for all $x,y,z\in \P$, if $x~R~z$ and $y\leq z$ then for all maximal $m$ such that $m\leq x$ and $m\leq y$, we have that $m~R~y$ (that is, $m~R~y$ for all pseudo-meets $m$ of $x$ and $y$).
    \end{enumerate}
\end{defn}
Below is a graphical representation of the restriction condition. The relation $R$ is indicated with the thick edges while $\leq$ is indicated with dashed edges. Here the top edge $x~R~z$ implies the bottom edge $m~R~y$. 
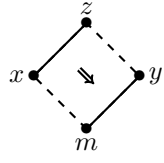
\begin{center}
    \begin{tikzpicture}[scale=0.7]
    \draw[thick, dashed] (0,0)--(1,1)--(0,2);
    \draw[thick, dashed] (0,0)--(-1,1)--(0,2);
    \draw (0,0) node{$\bullet$} node[below]{$m$};
    \draw (-1,1) node{$\bullet$} node[left]{$x$};
    \draw (1,1) node{$\bullet$} node[right]{$y$};
    \draw (0,2) node{$\bullet$} node[above]{$z$};
    \draw[thick] (-1,1)--(0,2);
    \draw[thick] (0,0)--(1,1);
    \draw[thick, double, -{Implies}, rotate=45] (0.7,0.9) -- (0.7,0.5);
    \end{tikzpicture}
\end{center}

\begin{remark}\label{rem:difres}
There are two versions of the restriction condition that appear in the literature. Sometimes the following is used in place of (2) in Definition~\ref{def:transf}:\smallskip
    
\noindent $(2')$ \emph{for all $x,y,z\in P$, if $x~R~z$ and $y\leq z$ then $(x\meet y)~R~y$, provided $x\meet y$ exists.}\smallskip
    
\noindent This version coincides with the definition of transfer systems as wide subcategories of a poset category that are closed under pullbacks, which is given in \cite[Definition 2.22]{modelstr} for example (see also \cite[Definition 4.1]{selfdual}). Here the connection is that pullbacks in a poset are exactly given by meets.

If $P$ is a meet semi-lattice, then the two versions of the restriction condition $(2)$ and $(2')$ are equivalent. But in general this is not the case. See the Hasse diagram below for an example:
    
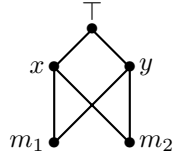
\begin{center}
    \begin{tikzpicture}
    \node (a) at (0,0) {$\bullet$};
    \node (b) at (1,0) {$\bullet$};
    \node (c) at (0,1) {$\bullet$};
    \node (d) at (1,1) {$\bullet$};
    \node (e) at (0.5,1.5) {$\bullet$};
    \draw[thick] (d.center)--(a.center)--(c.center)--(e.center)--(d.center)--(b.center)--(c.center);
    \draw (a) node[left]{$m_1$};
    \draw (b) node[right]{$m_2$};
    \draw (c) node[left]{$x$};
    \draw (d) node[right]{$y$};
    \draw (e) node[above]{$\top$};
    \end{tikzpicture}  
    \end{center}
Under condition $(2)$, the relation $x~R~\top$ in a transfer system would imply $m_1~R~y$ and $m_2~R~y$. However under condition $(2')$, $x~R~\top$ does not force any lower relations.

The restriction condition (2) coincides with the definition given in \cite[Definition 2.3]{BMO2}, for example. The authors note \cite[Remark 2.4]{BMO2} that this version should give a bijection between transfer systems and weak factorization systems for any finite poset, which was shown for lattices in \cite{selfdual}. We flesh this out in Section~\ref{sec:op} in Theorem~\ref{thm:wfs}. In this sense, $(2)$ is the ``correct'' definition. Also, as shown in Corollary~\ref{cor:oppable} we have $\Tr(\P^{\op})\cong \Tr(\P)^{\op}$ for any finite poset $\P$. But this isomorphism is false in general if we use $(2')$ instead of $(2)$---the poset given above provides a counterexample. See Remark~\ref{rem:difresag} for more on these differences. In this paper we will always use (2) as the restriction condition.
\end{remark}

For relations on arborescences, satisfying the restriction condition is equivalent to satisfying a simpler condition that can be described as saying ``if you can travel far, then you can travel close''. This is shown below where the dashed line indicates the given relation $\leq$ in $\A$ while the solid lines indicate the relation $R$:

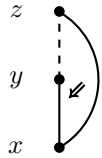
\begin{center}
    \begin{tikzpicture}[scale=0.9]
      \draw[thick] (0,-2)--(0,-1);
      \draw[thick] (0,-2) arc (-60:60:1.15);
      \draw[fill] (0,0) circle (2pt) node[left=10pt] {$z$};
      \draw[fill] (0,-1) circle (2pt) node[left=10pt] {$y$};
      \draw[fill]  (0,-2) circle (2pt) node[left=10pt] {$x$};
      \draw[thick, dashed] (0,-1)--(0,0);
      \draw[thick, double, -{Implies}, rotate=315] (1,-.5) -- (1,-.8);
    \end{tikzpicture}
\end{center}
The next proposition makes this precise. 
\begin{prop}
Let $R$ be a partial order on an arborescence $\A$. Then $R$ is a transfer system if and only if $R$ refines $\leq$ and satisfies
    \begin{center} \emph{(tree restriction condition)}  for all $x,y,z\in \A$, if $x\leq y \leq z \text{ and } x~R~z$, then $x~R~y.$
    \end{center}
\end{prop}
\begin{proof}
Note arborescences are meet semi-lattices, so we can use version $(2')$ of the restriction condition instead of $(2)$ to simplify the argument.
    
Let $R$ be a transfer system on $\A$. Then by definition $R$ refines $\leq$. Suppose we have $x\leq y \leq z$ and $x~R~z$. By the restriction condition $(x\meet y)~R~y$, but since $x\leq y$ we have $x\meet y =x$. Thus $x~R~y$.

Now suppose $R$ is a partial order that refines $\leq$ and satisfies the tree restriction condition. We just need to verify $R$ satisfies the usual restriction condition, so suppose we have $x~R~z$ and $y\leq z$. Note $x~R~z$ implies $x\leq z$, so by the definition of arborescence, $x\leq y$ or $y \leq x$. If $x\leq y$ then  $x~R~y$ by the tree restriction condition. Since $x\meet y =x$ we have $(x\meet y)~R~y$. If $y\leq x$ then $y\meet x = y$. The partial order $R$ is reflexive so $y~R~y$, which implies $(x\meet y)~R~y$. Thus $R$ satisfies the restriction condition and is a transfer system on $\A$. 
\end{proof}

We write $\Tr(\P)$ for the set of transfer systems on a poset $\P$. The set $\Tr(\P)$ is also a poset where the ordering is given by refinement, i.e. $R_1\leq R_2$ if and only if $R_1\subseteq R_2$. This poset is actually a lattice: meets are given by intersections, and joins are given by taking the smallest transfer system that contains the union. This is the transfer system generated by the union. See \cite[Construction A.1]{steiner} for more on generating transfer systems. We summarize this result below.

\begin{prop}
Let $\P$ be a finite poset. The set of transfer systems $\Tr(\P)$ is a lattice under refinement.
\end{prop}

\begin{example}
Consider the poset $\{0,1,2\}$ with ordering $0<1<2$. One can check there are exactly $5$ transfer systems, and the lattice of transfer systems is shown in Figure~\ref{fig:2ex}.
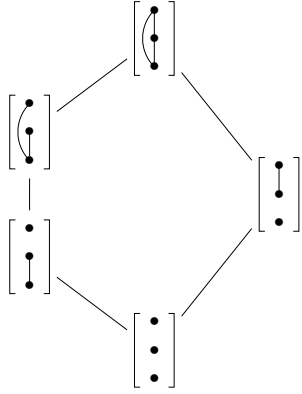
\begin{figure}[ht]
    \scalebox{0.75}{
    \begin{tikzpicture}[scale=1.1]
    \node (00) at (0,0) {\begin{tikzpicture}[scale=0.5]
        \node (2) at (0,2) {$\bullet$};
        \node (1) at (0,1) {$\bullet$};
        \node (0) at (0,0) {$\bullet$};
        \draw (-0.5,2.3)--(-0.7,2.3)--(-0.7,-0.3)--(-0.5,-0.3);
        \draw (0.5,2.3)--(0.7,2.3)--(0.7,-0.3)--(0.5,-0.3);
    \end{tikzpicture}};
    \node (01) at (-2,1.5) {\begin{tikzpicture}[scale=0.5]
        \node (2) at (0,2) {$\bullet$};
        \node (1) at (0,1) {$\bullet$};
        \node (0) at (0,0) {$\bullet$};
        \draw (0.center)--(1.center);
        \draw (-0.5,2.3)--(-0.7,2.3)--(-0.7,-0.3)--(-0.5,-0.3);
        \draw (0.5,2.3)--(0.7,2.3)--(0.7,-0.3)--(0.5,-0.3);
    \end{tikzpicture}};
    \node (02) at (-2,3.5) {\begin{tikzpicture}[scale=0.5]
        \node (2) at (0,2) {$\bullet$};
        \node (1) at (0,1) {$\bullet$};
        \node (0) at (0,0) {$\bullet$};
        \draw (0.center)--(1.center);
        \draw (0.center) edge[bend left=45] (2.center);
        \draw (-0.5,2.3)--(-0.7,2.3)--(-0.7,-0.3)--(-0.5,-0.3);
        \draw (0.5,2.3)--(0.7,2.3)--(0.7,-0.3)--(0.5,-0.3);
    \end{tikzpicture}};
    \node (11) at (2,2.5) {\begin{tikzpicture}[scale=0.5]
        \node (2) at (0,2) {$\bullet$};
        \node (1) at (0,1) {$\bullet$};
        \node (0) at (0,0) {$\bullet$};
        \draw (1.center)--(2.center);
        \draw (-0.5,2.3)--(-0.7,2.3)--(-0.7,-0.3)--(-0.5,-0.3);
        \draw (0.5,2.3)--(0.7,2.3)--(0.7,-0.3)--(0.5,-0.3);
    \end{tikzpicture}};
    \node (03) at (0,5) {\begin{tikzpicture}[scale=0.5]
        \node (2) at (0,2) {$\bullet$};
        \node (1) at (0,1) {$\bullet$};
        \node (0) at (0,0) {$\bullet$};
        \draw (0.center)--(1.center);
        \draw (1.center)--(2.center);
        \draw (0.center) edge[bend left=45] (2.center);
        \draw (-0.5,2.3)--(-0.7,2.3)--(-0.7,-0.3)--(-0.5,-0.3);
        \draw (0.5,2.3)--(0.7,2.3)--(0.7,-0.3)--(0.5,-0.3);
    \end{tikzpicture}};
    \draw (00)--(01)--(02)--(03)--(11)--(00);
    \end{tikzpicture}}
    \caption{The lattice of transfer systems for $[2]=\{0,1,2\}$.}
    \label{fig:2ex}
\end{figure}
\end{example}

Generalizing the above example, we write $[n]$ for the totally ordered poset $0<1<\cdots< n$. We call a poset a \textbf{chain} if it is isomorphic to $[n]$ for some $n$. In previous work of Balchin, Barnes, and Roitzheim \cite{BBR1}, they enumerated and classified the transfer system lattice for all totally ordered chains. These are counted by the \textbf{Catalan numbers}, which we briefly recall here. We denote the $n$th Catalan number by Cat$(n)$. This sequence appears in a wide variety of places in mathematics. As just some examples, the $n$th Catalan number counts
\begin{itemize}[leftmargin=*]
    \item the number of valid expressions containing $n$ pairs of parentheses that are correctly matched;
    \item the number of functions $f\colon [n-1]\to[n-1]$ that are order-preserving and satisfy $f(k)\leq k$ for all $k$; and
    \item the number of Dyck paths from $(0,0)$ to $(n,n)$, which are North-East lattice paths that don't go above the diagonal line $y=x$. 
\end{itemize}
The Catalan numbers have a closed form
\[
\text{Cat}(n)=\frac{1}{n+1}\binom{2n}{n},
\]
and they also satisfy the recurrence
\[
\text{Cat}(0)=1, \quad \text{Cat}(n)=\sum_{i=1}^{n-1}\text{Cat}(i-1)\text{Cat}(n-i) \quad \text{ for } n\geq 1.
\]
The first handful of Catalan numbers are given by
\[
1,2,5,14,42,132,429,\dots
\]
There is also a nice connection between Catalan numbers and transfer systems on chains:

\begin{thm}{BBR1}
Let $[n]$ denote the totally ordered set $0<1<\dots<n$. Then the number of transfer systems on $[n]$ is $\Cat(n+1)$, and furthermore, the lattice $\Tr([n])$ is isomorphic to the Tamari lattice. 
\end{thm}
In their proof, they show the number of transfer systems satisfies the Catalan recurrence by constructing an operation
\[
\odot \colon \Tr([i-1])\times \Tr[(n-i-1)]\to \Tr([n]).
\]
In \cite[Section 5]{selfdual} the authors reprove this result by establishing a bijection between transfer systems on $[n]$ and non-crossing partitions of $\{0,1,\dots, n\}$. In this paper, we provide two other proofs, one using the bijection with order-preserving functions (Corollary~\ref{cor:catalan}) and another using the recurrence relation from Theorem~\ref{thm:recur} and the Catalan triangle (Example~\ref{ex:catalan}).

\section{Ideals and pasting transfer systems}\label{sec:ideals}

Our approach to analyzing transfer systems on trees will be to paste together relations on subposets to get a transfer system on the entire poset. We prove a general pasting process for transfer systems on posets (not just trees). In order for our pasting to work, we need to consider posets that are ideals:

\begin{defn}
Let \(\P\) be a poset. A subset $I\subseteq \P$ is an (order) \textbf{ideal} of \(\P\) if it is downwards-closed (that is, for all $s,t\in \P$, if \(s\in I\) and \(t\leq s\) then \(t\in I\)).
\end{defn}

Observe $\P$ and $\emptyset$ are always ideals of $\P$. Similar to ideals in rings, we also have a notion of ``ideal generated by an element(s)''.

\begin{defn}
Suppose $\P$ is a finite poset and $S$ is a subset of $\P$. The \textbf{ideal generated by $S$} is given by \[\langle S\rangle := \{t\in P \mid t \leq s \text{ for some } s\in S\}.\]
If $S=\{x_1,\dots, x_n\}$ then we'll simply write $\langle S \rangle$ as $\langle x_1,\dots, x_n\rangle$. We say $I$ is \textbf{principal} if there exists an element $s\in \P$ such that $I=\langle s \rangle$.
\end{defn}

\begin{example}
Consider the poset $Y$ whose Hasse diagram is shown below.
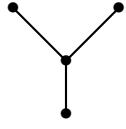
\begin{center}
    \begin{tikzpicture}[scale=0.7]
        \node (12) at (-1,2) {$\bullet$};
        \node (22) at (1,2) {$\bullet$};
        \node (11) at (0,1) {$\bullet$};
        \node (00) at (0,0) {$\bullet$};
        \draw[thick] (12.center)--(11.center)--(22.center);
        \draw[thick] (00.center)--(11.center);
    \end{tikzpicture}
\end{center}
The proper ideals of $Y$ consist of two copies of the three element chain, one copy of the two element chain, the singleton, and the empty set.
\end{example}

\begin{defn}
Let $I$ be an ideal of $\P$ and suppose $R\in \Tr(\P)$. The \textbf{restriction} of $R$ to $I$ is the relation $\rst{\P}{I}{R}$ (also denoted $R|_I$) on $I$ given by $x~(R|_I)~y$ if and only if $x~R~y$ for $x,y\in I$. 
\end{defn}
The following proposition is straightforward from the definitions.
\begin{prop}
  If \(R\in \Tr(\P)\) and \(I\) is an ideal of \(P\), then \(\rst{\P}{I}{R}\) is a transfer system on \(I\).
\end{prop}

Given ideals \(I, J\) of \(\P\) and transfer systems \(R\in \Tr(I)\) and \(S\in \Tr(J)\), we'd like to know when \(R\) and \(S\) can be pasted together to get a transfer system \(T\in \Tr(I\cup J)\). Of course this is only possible when \(R\) and \(S\) agree on the intersection \(I\cap J\). We prove that this agreement is a sufficient condition and that \(T\) is unique.

\begin{lemma}
\label{lem:icupj}
If \(R\in \Tr(I)\), \(S\in \Tr(J)\), and \(\rst{I}{I\cap J}{R} = \rst{J}{I\cap J}{S}\), then there exists unique \(T\in \Tr(I\cup J)\) such that \(\rst{I\cup J}{I}{T} = R\) and \(\rst{I\cup J}{J}{T} = S\).
\end{lemma}

\begin{proof}
Fix \(R\in \Tr(I)\) and \(S\in \Tr(J)\) such that \(\rst{I}{I\cap J}{R} = \rst{J}{I\cap J}{S}\). Define \(T\) as the relation on \(I\cup J\) given by $T=R\cup S$ (here we are viewing $R$ and $S$ as subsets of $(I\cup J)\times (I\cup J)$). Explicitly for $x, y\in I\cup J$, we have that $x~T~y$ if and only if either $x,y\in I$ and $x~R~y$ or $x,y\in J$ and $x~S~y$. Since $R$ and $S$ agree on $I\cap J$, we have that $T|_I=R$ and $T|_J=S$. We next verify $T$ is a transfer system on $I\cup J$.

First suppose $x~T~y$. Then $x,y\in I$ and $x~R~y$ or $x,y\in J$ and $x~S~y$. In either case $x\leq y$ because $R$ and $S$ are transfer systems on $I$ and $J$, respectively. Thus $T$ refines $\leq$. Reflexivity is proven similarly.
  
For transitivity, suppose \(x~T~y\) and \(y~T~z\). In the case that $x,y,z\in I$ or $x,y,z\in J$, we get that $x~T~z$ by the transitivity of $R$ or $S$, respectively. So assume, without loss of generality, that \(x,y\in I\) and \(y,z\in J\) (so we have $x~R~y$ and $y~S~z$). Then we get \(x\leq y\leq z\) since \(R,S\) are refinements of $\leq$. Since \(J\) is an ideal, this means \(x\in J\) as well. The relations $R$ and $S$ agree on $I\cap J$, so $x~R~y$ implies $x~S~y$. Finally transitivity of $S$ gives $x~T~z$.

To check the restriction condition, suppose \(x~T~z\) and \(y\in I \cup J\) satisfies \(y \leq z\). Consider the case that \(x,z\in I\) so $x~R~z$. Then \(y\in I\) because $I$ is an ideal and so \(m~R~y\) for any pseudo-meet $m$ which implies \(m~T~y\). A similar argument works if \(x,z\in J\). We have thus shown that \(T\) is a transfer system on \(I\cup J\). 
  
Finally for uniqueness fix a transfer system \(T'\in \Tr(I\cup J)\) such that \(\rst{I\cup J}{I}{T'} = R\) and \(\rst{I\cup J}{J}{T'} = S\). Suppose $x,y\in I\cup J$ such that $x~T'~y$. Since $T'$ is a transfer system, $x\leq y$. Without loss of generality suppose $y\in I$. Then $x\in I$ because $I$ is an ideal. Since $T'|_I=R$ we know $x~R~y$ which implies $x~T~y$. Thus $T'\subseteq T$ and we can similarly show $T\subseteq T'$. We conclude $T=T'$. 
\end{proof}

We can generalize to the setting where we have multiple ideals and transfer systems, as long as the transfer systems agree on the pairwise intersections.

\begin{thm}\label{thm:pasting}
Suppose $I_1, \dots, I_n$ are ideals of a poset $\P$ and that we have a transfer system $R_j$ on $I_j$ for each $1\leq j \leq n$. If the ideals agree on all pairwise intersections (that is, if $R_j|_{I_j\cap I_k} = R_k|_{I_j \cap I_k}$ for all $1\leq j < k \leq n$), then there exists a unique transfer system $T$ on $I_1\cup \dots \cup I_n$ such that $T|_{I_j}=R_{j}$ for all $j$.
\end{thm}
\begin{proof}
This follows from induction and Lemma~\ref{lem:icupj}. 
\end{proof}

\section{A bijection with order-preserving functions} \label{sec:bijmono}
In this section we establish a bijection between order-preserving functions and transfer systems on arborescences. We start by investigating chains and then use Theorem~\ref{thm:pasting} to generalize to trees. 

\begin{defn}
Let \(R\in\Tr([n])\). Define a function \(\delta_n(R) \colon [n]\to \N\) by
\[ k\mapsto \#\{ i\in [n] \colon i~R~k\} - 1 = \#\{i\in [k-1]\colon r~R~k\}.\]
We call this the \textbf{down-labeling} of $[n]$ associated to $R$. We thus have a function \(\delta_n\colon \Tr([n])\to \N^{[n]}\), \(R\mapsto \delta_n(R).\)
\end{defn}
Observe $\delta_n(R)(k)$ is the number of non-reflexive edges flowing into vertex $k$ when we draw the graph with edges for all of the relation $R$ (though we will always draw graphs without the reflexive loops). Figure~\ref{fig:2dl} shows the five transfer systems on $[2]$ as well as their down-labelings. 
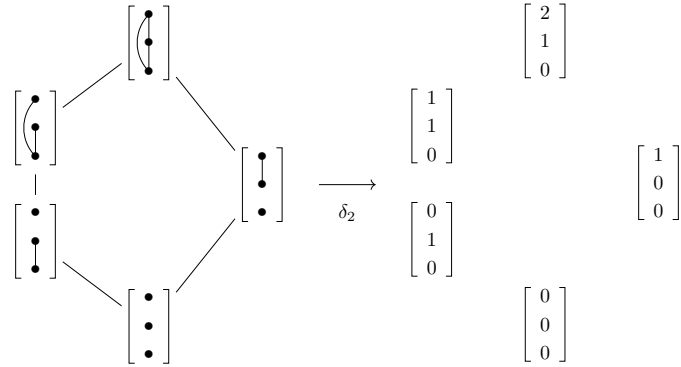
\begin{figure}[ht]
    \scalebox{0.75}{\begin{tikzpicture}
    \node (00) at (0,0) {\begin{tikzpicture}[scale=0.5]
        \node (2) at (0,2) {$\bullet$};
        \node (1) at (0,1) {$\bullet$};
        \node (0) at (0,0) {$\bullet$};
        \draw (-0.5,2.3)--(-0.7,2.3)--(-0.7,-0.3)--(-0.5,-0.3);
        \draw (0.5,2.3)--(0.7,2.3)--(0.7,-0.3)--(0.5,-0.3);
    \end{tikzpicture}};
    \node (01) at (-2,1.5) {\begin{tikzpicture}[scale=0.5]
        \node (2) at (0,2) {$\bullet$};
        \node (1) at (0,1) {$\bullet$};
        \node (0) at (0,0) {$\bullet$};
        \draw (0.center)--(1.center);
        \draw (-0.5,2.3)--(-0.7,2.3)--(-0.7,-0.3)--(-0.5,-0.3);
        \draw (0.5,2.3)--(0.7,2.3)--(0.7,-0.3)--(0.5,-0.3);
    \end{tikzpicture}};
    \node (02) at (-2,3.5) {\begin{tikzpicture}[scale=0.5]
        \node (2) at (0,2) {$\bullet$};
        \node (1) at (0,1) {$\bullet$};
        \node (0) at (0,0) {$\bullet$};
        \draw (0.center)--(1.center);
        \draw (0.center) edge[bend left=45] (2.center);
        \draw (-0.5,2.3)--(-0.7,2.3)--(-0.7,-0.3)--(-0.5,-0.3);
        \draw (0.5,2.3)--(0.7,2.3)--(0.7,-0.3)--(0.5,-0.3);
    \end{tikzpicture}};
    \node (11) at (2,2.5) {\begin{tikzpicture}[scale=0.5]
        \node (2) at (0,2) {$\bullet$};
        \node (1) at (0,1) {$\bullet$};
        \node (0) at (0,0) {$\bullet$};
        \draw (1.center)--(2.center);
        \draw (-0.5,2.3)--(-0.7,2.3)--(-0.7,-0.3)--(-0.5,-0.3);
        \draw (0.5,2.3)--(0.7,2.3)--(0.7,-0.3)--(0.5,-0.3);
    \end{tikzpicture}};
    \node (03) at (0,5) {\begin{tikzpicture}[scale=0.5]
        \node (2) at (0,2) {$\bullet$};
        \node (1) at (0,1) {$\bullet$};
        \node (0) at (0,0) {$\bullet$};
        \draw (0.center)--(1.center);
        \draw (1.center)--(2.center);
        \draw (0.center) edge[bend left=45] (2.center);
        \draw (-0.5,2.3)--(-0.7,2.3)--(-0.7,-0.3)--(-0.5,-0.3);
        \draw (0.5,2.3)--(0.7,2.3)--(0.7,-0.3)--(0.5,-0.3);
    \end{tikzpicture}};
    \draw (00)--(01)--(02)--(03)--(11)--(00);
    \draw[->] (3,2.5)--(4,2.5);
    \draw (3.5,2) node{$\delta_2$};
    \end{tikzpicture}\hspace{0.2in}
    \begin{tikzpicture}
    \node (00) at (0,0) {\begin{tikzpicture}[scale=0.5]
        \node (2) at (0,2) {$0$};
        \node (1) at (0,1) {$0$};
        \node (0) at (0,0) {$0$};
        \draw (-0.5,2.3)--(-0.7,2.3)--(-0.7,-0.3)--(-0.5,-0.3);
        \draw (0.5,2.3)--(0.7,2.3)--(0.7,-0.3)--(0.5,-0.3);
    \end{tikzpicture}};
    \node (01) at (-2,1.5) {\begin{tikzpicture}[scale=0.5]
        \node (2) at (0,2) {$0$};
        \node (1) at (0,1) {$1$};
        \node (0) at (0,0) {$0$};
        \draw (-0.5,2.3)--(-0.7,2.3)--(-0.7,-0.3)--(-0.5,-0.3);
        \draw (0.5,2.3)--(0.7,2.3)--(0.7,-0.3)--(0.5,-0.3);
    \end{tikzpicture}};
    \node (02) at (-2,3.5) {\begin{tikzpicture}[scale=0.5]
        \node (2) at (0,2) {$1$};
        \node (1) at (0,1) {$1$};
        \node (0) at (0,0) {$0$};
        \draw (-0.5,2.3)--(-0.7,2.3)--(-0.7,-0.3)--(-0.5,-0.3);
        \draw (0.5,2.3)--(0.7,2.3)--(0.7,-0.3)--(0.5,-0.3);
    \end{tikzpicture}};
    \node (11) at (2,2.5) {\begin{tikzpicture}[scale=0.5]
        \node (2) at (0,2) {$1$};
        \node (1) at (0,1) {$0$};
        \node (0) at (0,0) {$0$};
        \draw (-0.5,2.3)--(-0.7,2.3)--(-0.7,-0.3)--(-0.5,-0.3);
        \draw (0.5,2.3)--(0.7,2.3)--(0.7,-0.3)--(0.5,-0.3);
    \end{tikzpicture}};
    \node (03) at (0,5) {\begin{tikzpicture}[scale=0.5]
        \node (2) at (0,2) {$2$};
        \node (1) at (0,1) {$1$};
        \node (0) at (0,0) {$0$};
        \draw (-0.5,2.3)--(-0.7,2.3)--(-0.7,-0.3)--(-0.5,-0.3);
        \draw (0.5,2.3)--(0.7,2.3)--(0.7,-0.3)--(0.5,-0.3);
    \end{tikzpicture}};
    \end{tikzpicture}}
    \caption{The down-labelings for the transfer systems on $[2]=\{0,1,2\}$.}
    \label{fig:2dl}
\end{figure}
Note $\delta_n(R)$ is neither order-preserving nor order-reversing in general, but we will use it to construct an order-preserving function shortly. First we show $\delta_n$ can be used to distinguish between two different transfer systems.
\begin{prop}\label{prop:dlisinj}
The function $\delta_n\colon \Tr([n])\to \N^{[n]}$ is injective. That is, two transfer systems have the same down-labeling if and only if they are equal. 
\end{prop}
\begin{proof}
We proceed by induction on $n$. First if $n=0$ then $\Tr([0])$ is a singleton so the result is immediate. Now fix $n\geq 0$ and suppose $\delta_n$ is injective. Our goal is to show $\delta_{n+1}$ is injective. Let $R,S\in \Tr([n+1])$ with $\delta_{n+1}(R)=\delta_{n+1}(S)$, and suppose to the contrary $R\neq S$. It is straightforward to check that then $\delta_n(R|_{[n]})=\delta_n(S|_{[n]})$, so by the inductive hypothesis we know $R|_{[n]} = S|_{[n]}$. Thus these relations only differ when considering which elements are related to the maximal element $n+1$.
    
Since $\delta_{n+1}(R) = \delta_{n+1}(S)$, we know the number of elements related to $n+1$ under $R$ is equal to the number related to $n+1$ under $S$. Thus $R\neq S$ implies we can fix $0 \leq k \leq n$ such that $k~R~(n+1)$ is true but $k~{S}~(n+1)$ is false, and we can fix $1\leq j \leq n$ such that $j~R~(n+1)$ is false but $j~{S}~(n+1)$ is true. Suppose without loss of generality that $j<k$. We produce a contradiction by showing $j~R~(n+1)$.
    
By the restriction condition, $j~S~(n+1)$ implies $j~S~n$ and since $R|_{[n]}=S|_{[n]}$ we also learn that $j~R~n$. Applying the restriction condition again we learn $j~R~k$. But now by transitivity, $j~R~k$ and $k~R~(n+1)$ implies $j~R~(n+1)$. 
\end{proof}

We now use $\delta_n(R)$ to define an order-preserving function $[n]\to \N$.

\begin{defn}
Let \(R\in\Tr([n])\). Define $\Delta_n(R)\colon [n]\to [n]$ to be the function defined by $k\mapsto k-\delta_n(R)(k)$. 
\end{defn}
Note $\Delta_n(R)$ is the zero function if and only if $R$ is the complete transfer system given by $R=~\leq$. Thus we can think of $\Delta_n(R)$ as measuring how far $R$ is from being complete. As an example we apply $\Delta_2$ to the five transfer systems in Figure~\ref{fig:2delta}. The arrows in the right diagram indicate the partial order on order-preserving functions: we say $f\leq g$ if and only if $f(k)\leq g(k)$ for all $k\in [n]$. As suggested by the figure, we will see that $\Delta_n$ is an order-reversing bijection to a family of functions. But it does not provide a poset isomorphism. Explicitly, there might be relations between functions  
that do not exist between the corresponding transfer systems. 
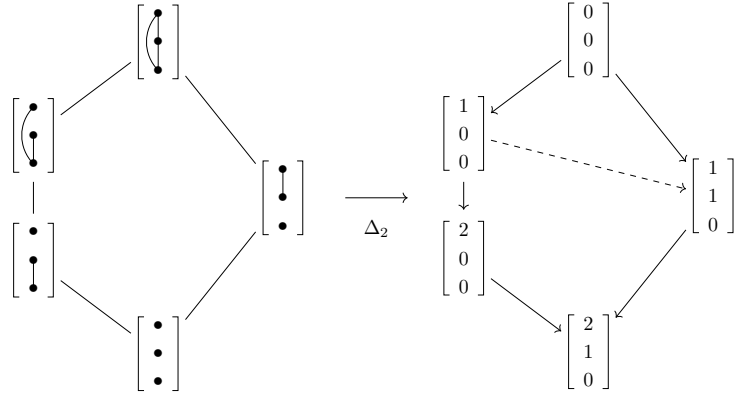
\begin{figure}[ht]
    \scalebox{0.75}{
    \begin{tikzpicture}[scale=1.1]
    \node (00) at (0,0) {\begin{tikzpicture}[scale=0.5]
        \node (2) at (0,2) {$\bullet$};
        \node (1) at (0,1) {$\bullet$};
        \node (0) at (0,0) {$\bullet$};
        \draw (-0.5,2.3)--(-0.7,2.3)--(-0.7,-0.3)--(-0.5,-0.3);
        \draw (0.5,2.3)--(0.7,2.3)--(0.7,-0.3)--(0.5,-0.3);
    \end{tikzpicture}};
    \node (01) at (-2,1.5) {\begin{tikzpicture}[scale=0.5]
        \node (2) at (0,2) {$\bullet$};
        \node (1) at (0,1) {$\bullet$};
        \node (0) at (0,0) {$\bullet$};
        \draw (0.center)--(1.center);
        \draw (-0.5,2.3)--(-0.7,2.3)--(-0.7,-0.3)--(-0.5,-0.3);
        \draw (0.5,2.3)--(0.7,2.3)--(0.7,-0.3)--(0.5,-0.3);
    \end{tikzpicture}};
    \node (02) at (-2,3.5) {\begin{tikzpicture}[scale=0.5]
        \node (2) at (0,2) {$\bullet$};
        \node (1) at (0,1) {$\bullet$};
        \node (0) at (0,0) {$\bullet$};
        \draw (0.center)--(1.center);
        \draw (0.center) edge[bend left=45] (2.center);
        \draw (-0.5,2.3)--(-0.7,2.3)--(-0.7,-0.3)--(-0.5,-0.3);
        \draw (0.5,2.3)--(0.7,2.3)--(0.7,-0.3)--(0.5,-0.3);
    \end{tikzpicture}};
    \node (11) at (2,2.5) {\begin{tikzpicture}[scale=0.5]
        \node (2) at (0,2) {$\bullet$};
        \node (1) at (0,1) {$\bullet$};
        \node (0) at (0,0) {$\bullet$};
        \draw (1.center)--(2.center);
        \draw (-0.5,2.3)--(-0.7,2.3)--(-0.7,-0.3)--(-0.5,-0.3);
        \draw (0.5,2.3)--(0.7,2.3)--(0.7,-0.3)--(0.5,-0.3);
    \end{tikzpicture}};
    \node (03) at (0,5) {\begin{tikzpicture}[scale=0.5]
        \node (2) at (0,2) {$\bullet$};
        \node (1) at (0,1) {$\bullet$};
        \node (0) at (0,0) {$\bullet$};
        \draw (0.center)--(1.center);
        \draw (1.center)--(2.center);
        \draw (0.center) edge[bend left=45] (2.center);
        \draw (-0.5,2.3)--(-0.7,2.3)--(-0.7,-0.3)--(-0.5,-0.3);
        \draw (0.5,2.3)--(0.7,2.3)--(0.7,-0.3)--(0.5,-0.3);
    \end{tikzpicture}};
    \draw (00)--(01)--(02)--(03)--(11)--(00);
    \draw[->] (3,2.5)--(4,2.5);
    \draw (3.5,2) node{$\Delta_2$};
    \end{tikzpicture}\hspace{0.2in}
    \begin{tikzpicture}[scale=1.1]
    \node (00) at (0,0) {\begin{tikzpicture}[scale=0.5]
        \node (2) at (0,2) {$2$};
        \node (1) at (0,1) {$1$};
        \node (0) at (0,0) {$0$};
        \draw (-0.5,2.3)--(-0.7,2.3)--(-0.7,-0.3)--(-0.5,-0.3);
        \draw (0.5,2.3)--(0.7,2.3)--(0.7,-0.3)--(0.5,-0.3);
    \end{tikzpicture}};
    \node (01) at (-2,1.5) {\begin{tikzpicture}[scale=0.5]
        \node (2) at (0,2) {$2$};
        \node (1) at (0,1) {$0$};
        \node (0) at (0,0) {$0$};
        \draw (-0.5,2.3)--(-0.7,2.3)--(-0.7,-0.3)--(-0.5,-0.3);
        \draw (0.5,2.3)--(0.7,2.3)--(0.7,-0.3)--(0.5,-0.3);
    \end{tikzpicture}};
    \node (02) at (-2,3.5) {\begin{tikzpicture}[scale=0.5]
        \node (2) at (0,2) {$1$};
        \node (1) at (0,1) {$0$};
        \node (0) at (0,0) {$0$};
        \draw (-0.5,2.3)--(-0.7,2.3)--(-0.7,-0.3)--(-0.5,-0.3);
        \draw (0.5,2.3)--(0.7,2.3)--(0.7,-0.3)--(0.5,-0.3);
    \end{tikzpicture}};
    \node (11) at (2,2.5) {\begin{tikzpicture}[scale=0.5]
        \node (2) at (0,2) {$1$};
        \node (1) at (0,1) {$1$};
        \node (0) at (0,0) {$0$};
        \draw (-0.5,2.3)--(-0.7,2.3)--(-0.7,-0.3)--(-0.5,-0.3);
        \draw (0.5,2.3)--(0.7,2.3)--(0.7,-0.3)--(0.5,-0.3);
    \end{tikzpicture}};
    \node (03) at (0,5) {\begin{tikzpicture}[scale=0.5]
        \node (2) at (0,2) {$0$};
        \node (1) at (0,1) {$0$};
        \node (0) at (0,0) {$0$};
        \draw (-0.5,2.3)--(-0.7,2.3)--(-0.7,-0.3)--(-0.5,-0.3);
        \draw (0.5,2.3)--(0.7,2.3)--(0.7,-0.3)--(0.5,-0.3);
    \end{tikzpicture}};
    \draw[->] (01)--(00);
    \draw[->] (02)--(01);
    \draw[->] (03)--(02);
    \draw[->] (11)--(00);
    \draw[->] (03)--(11);
    \draw[dashed, ->] (02)--(11);
    \end{tikzpicture}}
    \caption{Applying $\Delta_2$ to the five transfer systems on $[2]$.}
    \label{fig:2delta}
\end{figure}

We have the following key properties of $\delta_n(R)$ and $\Delta_n(R)$.

\begin{prop}\label{pr:dlchains}
Let $R\in\Tr([n])$. Then $\delta_n(R)(0)=0$ and \[\delta_n(R)(k+1)\leq \delta_n(R)(k)+1\] for all $0\leq k \leq n-1$. Thus the function $\Delta_n(R)\colon [n]\to [n]$ is order-preserving (i.e. $j\leq k$ implies $\Delta_n(R)(j)\leq \Delta_n(R)(k)$) and satisfies $\Delta_n(R)(k)\leq k$ for all $k$. 
\end{prop}
\begin{proof} It follows immediately from the definition that $\delta_n(R)(0)=0$. Next observe $\delta_n(R)(k+1)$ counts the number of elements in the set $\{j\in [k]:j~R~(k+1)\}$ while $\delta_n(R)(k)+1$ counts the number of elements in the set $\{j\in [k]: j~R~k\}$. Thus it suffices to show 
\[\{j\in [k]: j~R~(k+1)\}\subseteq \{j\in [k]: j~R~k\},\]
but this follows immediately from the restriction condition.

For the statements about $\Delta_n$, observe
\begin{align*}
\Delta_n(R)(k)&=k-\delta_n(R)(k)\\
&= (k+1) - (\delta_n(R)(k)+1) \\
&\leq (k+1)-\delta_n(R)(k+1)\\
&=\Delta_n(R)(k+1).
\end{align*}
Hence $\Delta_n(R)$ is order-preserving. Finally $\Delta_n(R)(k)\leq k$ because $\delta_n(R)$ is nonnegative.
\end{proof}

We now prove there is an order-reversing bijection between transfer systems and order-preserving functions bounded by the identity for chains. 

\begin{thm}\label{thm:chainmonotone}
Let $M_n$ denote the set of order-preserving functions $f\colon [n]\to [n]$ that satisfy $f(k)\leq k$ for all $k$. Then $\Delta_n\colon \Tr([n])\to M_n$ is an order-reversing bijection, where in $M_n$ we have $f\leq g$ if and only if $f(k)\leq g(k)$ for all $k$.
\end{thm}
\begin{proof}
We showed $\delta_n$ is injective in Proposition~\ref{prop:dlisinj} and thus $\Delta_n$ is injective since it is the composition of two injective functions.

We prove surjectivity using induction on $n$. First observe $M_0$ and $\Tr([0])$ each only contain one element, so surjectivity follows.

Now fix $n\geq 0$ and suppose for every function $f\in M_n$ we can find a transfer system $R\in \Tr([n])$ such that $\Delta_n(R)=f$. Let $f\in M_{n+1}$. We can restrict the domain of $f$ to get $f|_{[n]}\colon [n]\to [n]$, and by definition of $M_{n+1}$ it is clear $f|_{[n]}\in M_n$. Let $R'\in \Tr(f)$ be the transfer system such that $\Delta_n(R')=f|_{[n]}$. 

We next construct a relation $R$ so that $R|_{[n]}=R'$. First for all $i,j\in [n]$ define $R$ so that $i~R~j$ if and only if $i~(R')~j$. It remains to specify the relations involving $n+1$.

Suppose $f(n+1)=m$. We want to define $R$ so that
\[
\Delta_{n+1}(R)(n+1)= m, \quad \text{i.e.} \quad \#\{i\in [n]\mid i~\cancel{R}~(n+1)\} =m, 
\]
or equivalently so that
\[
\delta_{n+1}(R)(n+1)=n+1-m, \quad \text{i.e.}\quad \#\{i\in [n]\mid i~R~(n+1)\}= n+1-m.
\]
Let $X=\{j\in [n-1] : j~R~n\}$. By the inductive hypothesis \[\#X =\delta_n(R')(n)= n-f(n).\] Since $f$ is order-preserving we also know $f(n)\leq f(n+1)=m$. We break into cases: $f(n)=m$ and $f(n)<m$.

First if $f(n)=m$ define $R$ so that $j~R~(n+1)$ if and only if $j\in X$, $j=n$, or $j=n+1$. Note we have  $\#\{i\in [n]\mid i~R~(n+1)\}=\#X+1=n-m+1$, as desired.

Next suppose $f(n)<m$. Since all outputs are integers and $f$ is monotone, observe $n+1-m \leq n-f(n)=\#X$. Thus we can fix the $n+1-m$ smallest elements $x_1,x_2, \dots x_{n+1-m}$ in $X$. Define $j~R~(n+1)$ if and only if $j\in \{x_1,x_2,\dots, x_{n+1-m}\}$ or $j=n+1$. 

Finally we check $R$ is a transfer system on $[n+1]$ in either case. The fact that $R$ is reflexive and refines $\leq$ follows from the inductive hypothesis and the construction of $R$. For transitivity, fix $x<y<z$ with relations $x~R~y$ and $y~R~z$. If $x,y,z\in [n]$ then $x~R~z$ by the inductive hypothesis. Thus suppose $z=n+1$. Note by construction of $R$, $y~R~(n+1)$ implies $y~R~n$. By transitivity of $R$ on $[n]$, we learn $x~R~n$. But since $x\leq y$ and we chose minimal elements in the paragraph above, we must have that $x~R~(n+1)$ as well. This verifies transitivity. 

For the tree restriction condition, fix relations $x~R~z$ and $x\leq y\leq z$. If $z\leq n$ then we know $x~R~y$ by the inductive hypothesis. Thus suppose $z=n+1$. Note if $y=z$ then $x~R~y$, so also suppose $x\leq y < z$. Now $x~R~(n+1)$ implies $x~R~n$ by the construction of $R$. By the restriction condition for $R|_{[n]}$ we thus know $x~R~y$. 

Finally we prove $\Delta_n$ is order-reversing. Fix $R, R'\in \Tr([n])$. Suppose $R'\leq R$. Then it is clear $\delta_n(R')(k)\leq \delta_n(R)(k)$ for all $k\in [n]$ since any relation that holds in $R'$ also holds in $R$. Thus $\Delta_n(R')(k)\geq \Delta_n(R)(k)$ for all $k$ and so $\Delta_n(R')\geq \Delta_n(R)$.
\end{proof}
As seen in Figure~\ref{fig:2delta}, while $\Delta_n$ is a map of posets $\Tr([n])^{\op}\to M_n$, it is not a poset isomorphism. The above theorem can be used to show transfer systems on chains are enumerated by the Catalan numbers, which was originally proven in \cite{BBR1}.

\begin{cor}\label{cor:catalan}
The number of transfer systems on $\Tr([n])$ is equal to $\Cat(n+1)$. 
\end{cor}
\begin{proof}
It is well-known that the number of order-preserving functions $f\colon [n]\to [n]$ that satisfy $f(k)\leq k$ for all $k$ is exactly Cat$(n+1)$, so the corollary  immediately follows from Theorem~\ref{thm:chainmonotone}. 
    
For the curious reader, we briefly outline why $\#M_n=\Cat(n+1)$ by producing a bijection from $M_{n}$ to the set of Dyck paths from $(0,0)$ to $(n+1, n+1)$---these are lattice paths that never cross above (but can touch) the diagonal $y=x$ and consist only of moves that are one unit east or one unit north. For a function $f\colon [n]\to[n]$, plot the lattice points $(k,f(k))$ for $k\in [n]$ as well as the point $(n+1, n+1)$. Now use these points to draw a Dyck path as follows. Start at $(0,0)$, move east to $(1,0)$, move north (if necessary) to $(1,f(1))$, move east to $(2,f(1))$, move north (possibly multiple units) to $(2,f(2))$, move east to $(3,f(2))$, and so on, until you reach $(n,f(n))$. Then move east to $(n+1,f(n))$ and finally north to $(n+1, n+1)$. See Figure~\ref{fig:dyckpaths} for two examples when $n=3$. We leave it as an exercise for the reader to construct the inverse.
    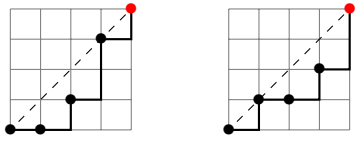
\begin{figure}[ht]
    \begin{center}
        \begin{tikzpicture}[scale=0.4]
            \draw[step=1cm, gray, thin] (0,0) grid (4,4);
            \draw[dashed] (-0.1,-0.1)--(4.1,4.1);
            \draw[thick] (0,0)--(1,0)--(2,0)--(2,1)--(3,1)--(3,3)--(4,3)--(4,4);
            \draw (0,0) node{$\bullet$};
            \draw (1,0) node{$\bullet$};
            \draw (2,1) node{$\bullet$};
            \draw (3,3) node{$\bullet$};
            \draw[red] (4,4) node{$\bullet$};
        \end{tikzpicture}\hspace{0.3in}
        \begin{tikzpicture}[scale=0.4]
            \draw[step=1cm, gray, thin] (0,0) grid (4,4);
            \draw[dashed] (-0.1,-0.1)--(4.1,4.1);
            \draw[thick] (0,0)--(1,0)--(1,1)--(2,1)--(3,1)--(3,2)--(4,2)--(4,4);
            \draw (0,0) node{$\bullet$};
            \draw (1,1) node{$\bullet$};
            \draw (2,1) node{$\bullet$};
            \draw (3,2) node{$\bullet$};
            \draw[red] (4,4) node{$\bullet$};
        \end{tikzpicture}
    \end{center}
    \caption{Dyck paths for ($0\mapsto 0$, $1\mapsto 0$, $2\mapsto 1$, $3\mapsto 3$) and ($0\mapsto 0$, $1\mapsto 1$, $2\mapsto 1$, $3\mapsto 2$).}
    \label{fig:dyckpaths}
    \end{figure}
\end{proof}

\begin{remark}\label{rem:stanlat}
The Stanley lattice and the Tamari lattice are two common orderings defined on Catalan objects. These are related in that the Tamari lattice is a proper refinement of the Stanley lattice. The collection of Dyck paths discussed in the above proof has a partial ordering given by $P_1\leq P_2$ if and only if the path given by $P_2$ is always above or on the path given by $P_1$. This ordering of Dyck paths gives rise to the Stanley lattice. The bijection with $M_n$ described above is easily seen to be a lattice isomorphism, so as a poset, $M_n$ is also the Stanley lattice. From \cite{BBR1} we know the lattice $\Tr([n])$ is given by the Tamari lattice. Thus it is consistent that our function $\Delta_n$ is order-reversing but does not give an isomorphism of lattices.
\end{remark}

We now generalize these constructions to arborescences using Theorem~\ref{thm:pasting}. We start by defining a similar notion of down-labeling and an associated order-preserving function. 
\begin{defn}
    Let $\mathcal{A}$ be an arborescence and suppose $R \in \Tr(\A)$. Define a function $\delta_{\A}(R)\colon\A\to \N$ by $\delta_\A(R)(x) = \#\{y\in \A ~|~ y~R~x\}-1$. We call the function $\delta_\A(R)$ the \textbf{down-labeling} of $\A$ associated the transfer system $R$. Define $\Delta_\A(R)\colon \A \to \N$ to be the function $\Delta_\A(R)(x)=\rho(x)-\delta_\A(x)$. 
\end{defn}
 
As with chains, we can think of $\Delta_\A(R)$ as measuring how far $R$ is from the complete transfer system. 
Figure~\ref{fig:lambda} shows an example of a transfer system $R$ on an aborscence $\A$ and the corresponding functions $\delta_\A(R)$ and $\Delta_\A(R)$.  

\begin{figure}[ht]
  \centering
  \scalebox{0.8}{
  \begin{tikzpicture}[scale=0.8, thick,level/.style={sibling distance=20mm/#1}]
    \node[inner sep=0] (root-a) at (0,0) {\large{$\bullet$}} [grow=up,thick]
    child {[fill] circle(2pt)
      child {[fill] circle(2pt)}
    }
    child {[fill] circle(2pt)
      child {[fill] circle(2pt)}
      child {[fill] circle(2pt)}
    }
    child {[fill] circle(2pt)
      child {[fill] circle(2pt)}
      child {[fill] circle(2pt)}
      child {[fill] circle(2pt)}
    };
  
    \node[inner sep=0] (root-t) at (7,0) {} [grow=up]
    child {[fill] circle(2pt)
      child [dashed] {[fill] circle (2pt)}
    }
    child [solid] {[fill] circle(2pt)
      child [dashed] {[fill] circle(2pt)}
      child [solid] {[fill] circle(2pt)}
    }
    child [dashed] {[fill] circle(2pt)
      child [solid] {[fill] circle(2pt)}
      child [solid] {[fill] circle(2pt)}
      child [solid] {[fill] circle(2pt)}
    };
    \draw[fill, thin] (root-t) circle (2pt);
    \draw (root-t) .. controls +(10mm,20mm) .. (root-t-1-1);
    \draw (root-t) .. controls +(-5mm,5mm) .. (root-t-2-2);

    \draw (root-t) node[anchor=north, below=3mm] {\(R\)};
    \draw (root-a) node[anchor=north, below=3mm] {\(\A\)};
  \end{tikzpicture}}\medskip
  
  \scalebox{0.8}{
  \begin{tikzpicture}[scale=0.8,thick,level/.style={sibling distance=20mm/#1}]
  \node (root-l) at (0,0) {0} [grow=up,thick]
    child {node {1}
      child {node {1}}
    }
    child {node {1}
      child {node {0}}
      child {node {2}}
    }
    child {node {0}
      child {node{1}}
      child {node{1}}
      child {node{1}}
    };
    \draw (root-l) node[anchor=north, below=3mm] {\(\delta_\A(R)\)};

  \node (root-l) at (7,0) {0} [grow=up,thick]
    child {node {0}
      child {node {1}}
    }
    child {node {0}
      child {node {2}}
      child {node {0}}
    }
    child {node {1}
      child {node{1}}
      child {node{1}}
      child {node{1}}
    };
    \draw (root-l) node[anchor=north, below=3mm] {\(\Delta_\A(R)\)};
  \end{tikzpicture}}
  \caption{Example of $\delta_\A(R)$ and $\Delta_\A(R)$.}
  \label{fig:lambda}
\end{figure}
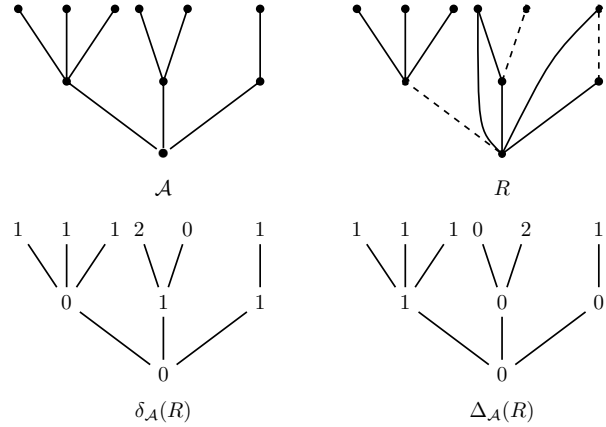

\begin{prop}
Let $\A$ be an arborescence and $R$ be a transfer system on $\A$. Then $\Delta_\A(R)\colon \A\to \N$ is an order-preserving function and $\Delta_\A(R) \leq \rho$.
\end{prop}
\begin{proof}
This is very similar to the proof of Proposition~\ref{pr:dlchains} so we just provide an outline. First $\Delta_\A(R)= \rho - \delta_\A(R) \leq \rho$ because $\delta_\A(R)\geq 0$. Suppose $x\leq y$ with $\rho(y)=\rho(x)+1$ (so $y$ covers $x$). Then we can use the restriction condition to show that $\delta_\A(y)\leq \delta_\A(x)+1$. Thus \[\Delta_\A(y)=\rho(y)-\delta_\A(y) \geq \rho(x)+1 - (\delta_\A(x)+1) = \Delta_\A(x).\]
\end{proof}

\begin{thm}\label{thm:arbij}
Let \(\mathcal{A}\) be an arborescence. Let \(M_\mathcal{A}\) be the set of functions
\[ \{f\colon \mathcal{A}\to \N~|~ \text{\(f\) is order-preserving and \(f \leq \rho\)}\}.\]
The map $\Delta_\A\colon \Tr(\A)\to M_\A$ is an order-reversing bijection.
\end{thm}

\begin{proof}
Suppose the leaf set of \(\mathcal{A}\) is \(\{x_1, \dots, x_n\}\). Then $\A=\langle x_1,\dots, x_n \rangle$. Note for each $1\leq j \leq n$, the principal ideal $\langle x_j \rangle$ is isomorphic to the chain $[\rho(x_j)]$. Let $m_j=\rho(x_j)$. We have a map $(-)|_{\langle x_j \rangle}\colon M_\A \to M_{m_j}$ given by restricting the domain to $\langle x_j \rangle$ and then identifying $\langle x_j \rangle$ with $[m_j]$. We have the following commutative diagram (in the vertical maps we are also using the unique poset isomorphism $\langle x_j \rangle \cong [m_j]$):
\begin{center}
  \begin{tikzcd}
      \Tr(\A)\ar[d,"\Res^\A_{\langle x_j \rangle}" left] \ar[r, "\Delta_\A"] & M_\A \ar[d,"(-)|_{\langle x_j \rangle}" right] \\
      \Tr([m_j]) \ar[r, "\Delta_{m_j}" below] & M_{m_j}~.
  \end{tikzcd}
\end{center}

We first show $\Delta_\A$ is injective. Fix $R,S\in \Tr(\A)$ and suppose $\Delta_\A(R)=\Delta_\A(S)$. Recall that $\Delta_{m_j}$ is injective for each $1\leq j \leq n$. So from the commutative diagram, we learn $\Res^\A_{\langle x_j \rangle}(R) = \Res^\A_{\langle x_j \rangle}(S)$. The ideals $\langle x_1 \rangle, \dots, \langle x_n\rangle$ union to give all of $\A$ and so by the uniqueness in Theorem~\ref{thm:pasting} we have that $R=S$.

We next show $\Delta_\A$ is surjective. Fix a monotone increasing function $\mu\colon \A\to \N$ with \(\mu \leq \rho\). For each $j$, $\mu|_{\langle x_j \rangle}$ is a monotone function. Note that the chain $\langle x_j \rangle$ is isomorphic to $[m_j]$ under the mapping $y\mapsto \rho(y)$. Using this identification of $\langle x_j \rangle$ with the chain $[m_j]$, we can use Theorem~\ref{thm:chainmonotone} to fix a transfer system $R_j$ on $\langle x_j \rangle$ so that $\Delta_{m_j}(R_j) = \mu|_{\langle x_j \rangle}$. Now from Theorem~\ref{thm:pasting}, there exists a transfer system $R$ on $\A$ so that $R|_{\langle x_j \rangle} = R_j$ for all $j$. Finally note by construction $\Delta_\A(R)= \mu$.

\end{proof}

\section{The recurrence relation}\label{sec:recur}
We now prove a recurrence for order-preserving functions bounded by the rank function. We then use the bijection from Theorem~\ref{thm:arbij} together with the recurrence to enumerate transfer systems in a handful of examples. We start with a definition. 

\begin{defn}
Let $\A$ be an arborescence and $x\in \A$. Write $x^\uparrow$ for the set $\{y\in \A ~|~ x\leq y\}$. For $n\in \N$ define
\(B(x,n)\) to be the set of order-preserving functions \(\sigma\colon x^{\uparrow}\to \N\) such that \(\sigma \leq \rho|_{x^{\uparrow}}\) and \(\sigma(x) \geq n\).
We denote the cardinality of $B(x,n)$ by $b(x,n)$.
\end{defn}
Note if $n>\rho(x)$ then $B(x,n)=\emptyset$. Also by Theorem~\ref{thm:arbij} we have that
\[
\#\Tr(\A)= b(\bot, 0).
\]
Our recurrence relation provides a formula for $b(x,n)$ in terms of $b(y,i)$ where $y$ ranges over all covers of $x$. This will terminate when we reach a leaf. We start with an easy lemma.
\begin{lemma}\label{lem:leafcount}
Suppose $\A$ is an arborescence and $x\in \A$ is a leaf. For $0\leq n\leq \rho(x)$ we have that
\[
b(x,n)=\rho(x)-n+1.
\]
\end{lemma}
\begin{proof}
Since $x$ is a leaf we have that $x^\uparrow=\{x\}$. Thus the set $B(x,n)$ consists of functions $\sigma\colon \{x\}\to \N$ such that $n\leq \sigma(x)\leq \rho(x)$. Such a function is determined by just picking an integer in the appropriate range, and there are exactly $\rho(x)-n+1$ integers to pick from.
\end{proof}
\begin{thm}\label{thm:recur}
Let $\A$ be an arborescence and $x\in \A$. For $y\in \A$, we write $y\gtrdot x$ for ``$y$ covers $x$''. Then we have
\[ b(x,n) = \sum_{i = n}^{\rho(x)} \prod_{y \gtrdot x} b(y,i).\]
\end{thm}
\begin{proof}
Fix \(x\in \mathcal{A}\) and \(n\leq \rho(x)\). We have two cases: $x$ is a leaf and $x$ is an internal node. First suppose $x$ is a leaf. It follows that $x$ has no covers and so
\[ \sum_{i=n}^{\rho(x)}\prod_{y \gtrdot x} b(y,i)= \sum_{i=n}^{\rho(x)}1 = \rho(x)-n+1.\]
By Lemma~\ref{lem:leafcount} we conclude \(b(x,n)=\sum_{i=n}^{\rho(x)}\prod_{y \gtrdot x} b(y,i)\) in the case that $x$ is a leaf.

Now suppose $x$ is not a leaf. Consider $\sigma \in B(x,n)$. Then $\sigma$ is a function \(\sigma\colon x^\uparrow\to \N\) such that \(\sigma \leq \rho|_{x^\uparrow}\) and \(\sigma(x)=i\) for some \(n\leq i \leq \rho(x)\). Also notice for each \(y \gtrdot x\) the function \(\sigma|_{y^\uparrow}\) must satisfy $\sigma|_{y^\uparrow}\leq \rho|_{y^\uparrow}$ and $\sigma|_{y^\uparrow}(y)\geq i$. Thus we have function
  \[
  \Phi\colon B(x,n) \to \coprod_{i=n}^{\rho(x)}\prod_{y\gtrdot x} B(y,i),\quad \sigma\mapsto (\sigma|_{y^\uparrow})_{y\gtrdot x}.
  \]
On the other hand, the set $x^\uparrow\setminus \{x\}$ is partitioned by $y^\uparrow$ as $y$ ranges over all covers of $x$ (this is because $\A$ is an arborescence). Thus for each $n\leq i \leq \rho(x)$, we get a map
  \[
  \prod_{y\gtrdot x} B(y,i) \to B(x,n),\quad (\sigma_y)_{y\gtrdot x} \mapsto \sigma
  \]
where $\sigma\colon x^\uparrow\to \N$ is defined by $\sigma(x)=i$ and for $z>x$, $\sigma(z)=\sigma_y(z)$ where $y$ is the unique cover of $x$ satisfying $x<y\leq z$. This gives rise to a map
    \[
  \Psi\colon \coprod_{i=n}^{\rho(x)}\prod_{y\gtrdot x} B(y,i)\to B(x,n) ,\quad (\sigma_{y})_{y\gtrdot x}\mapsto \sigma.
  \]
Observe $\Phi\circ \Psi=\id$ and $\Psi \circ \Phi =\id$. We conclude that
  \[
  b(x,n)=\left| \coprod_{i=n}^{\rho(x)}\prod_{y\gtrdot x} B(y,i)\right| = \sum_{i=n}^{\rho(x)} \prod_{y \gtrdot x} b(y,i). 
  \]
\end{proof}

\begin{example}\label{ex:depth2}
Consider the arborescence whose Hasse diagram is the tree shown below. 
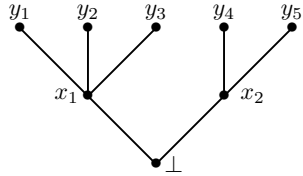
\begin{center}
    \scalebox{0.9}{\begin{tikzpicture}
        \node (00) at (0,0) {$\bullet$};
        \node (10) at (-1,1) {$\bullet$};
        \node (11) at (1,1) {$\bullet$};
        \node (20) at (-2,2) {$\bullet$};
        \node (21) at (-1,2) {$\bullet$};
        \node (22) at (0,2) {$\bullet$};
        \node (23) at (1,2) {$\bullet$};
        \node (24) at (2,2) {$\bullet$};
        \draw[thick] (00.center)--(10.center)--(20.center);
        \draw[thick] (10.center)--(21.center);
        \draw[thick] (10.center)--(22.center);
        \draw[thick] (00.center)--(11.center)--(23.center);
        \draw[thick] (11.center)--(24.center);
        \draw[thick] (10) node[left]{$x_1$};
        \draw (00) node[right]{$\bot$};
        \draw (11) node[right]{$~x_2$};
        \draw (20) node[above]{$y_1$};
        \draw (21) node[above]{$y_2$};
        \draw (22) node[above]{$y_3$};
        \draw (23) node[above]{$y_4$};
        \draw (24) node[above]{$y_5$};
    \end{tikzpicture}}
\end{center}
Note $\rho(y_i)=2$ while $\rho(x_j)=1$. Using our recurrence and Lemma~\ref{lem:leafcount} we have that
\begin{align*}
\# \Tr(\A) &= b(\bot, 0)\\
&= b(x_1,0)\cdot b(x_2,0) \\
&= \left(b(y_1,0)b(y_2,0)b(y_3,0) + b(y_1,1)b(y_2,1)b(y_3,1) \right)\\ &\quad \cdot\left(b(y_4,0)b(y_5,0) +b(y_4,1)b(y_5,1) \right)\\
&=(3^3+2^3)(3^2+2^2)\\
&=455.
\end{align*}

More generally let $\A$ be an arborescence whose Hasse diagram is a tree of depth $2$, that is, $\rho(\A)=\{0,1,2\}$. Suppose $\A$ has $k$ elements $x_1, \dots,x_k$ of depth one, and for each $i$, $x_i$ has $j_i$ children. Then the same method shows
    \[
    \# \Tr(\A)= \prod_{i=1}^k (3^{j_i}+2^{j_i}).
    \]
\end{example}

\begin{example}\label{ex:catalan}
Recall $[n]$ denotes the ordered chain $0<1<\cdots<n$. From \cite{BBR1} we know transfer systems on $[n]$ are enumerated by the Catalan numbers. Explicitly \[\#\Tr([n])=\text{Cat}(n+1)=\frac{1}{n+2}\binom{2n+2}{n+1}.\] Since $[n]$ is an arborescence, we can use the recurrence from Theorem~\ref{thm:recur} to recover some interesting relations among the Catalan numbers.

We start by computing $b(\bot, 0)=b(0,0)$ for $[n]$ for small values of $n$. Note that each vertex is either a leaf or has exactly one child, so we do not have any products appearing from the recurrence relation. Also observe for $0\leq k\leq j<n$, Theorem~\ref{thm:recur} and that $\rho(j)=j$ gives us
    \begin{equation}\label{eq:chainsum}
    b(j,k)=\sum_{i=k}^{j}b(j+1,i)=b(j+1,k)+b(j+1,k+1)+\dots+b(j+1,j).
    \end{equation}
We use this formula together with the fact that $b(n,k)=n-k+1$ in $[n]$ to recursively compute in cases:
    \begin{align*}
    n=1\colon & \quad b(0,0)=b(1,0)=2.\\
    n=2\colon &  \quad b(0,0)=b(1,0)=b(2,0)+b(2,1) = 5.\\
    n=3\colon &\quad b(0,0) =b(2,0)+b(2,1)\\
    &\quad \quad \quad \hspace{0.09in}= (b(3,0)+b(3,1)+b(3,2)) + (b(3,1)+b(3,2))\\
    &\quad \quad \quad \hspace{0.09in}=b(3,0)+2b(3,1)+2b(3,2) = 14.\\
    n=4\colon& \quad b(0,0)=b(4,0)+3b(4,1)+5b(4,2)+5b(4,3) = 42.
    \end{align*} 
Notice the following pattern: if we want to find the coefficient of $b(n,k)$ in the formula for $b(0,0)$ for $[n]$, we sum up the coefficients of $b(n-1,i)$ for $i\leq k$ in the expression for $[n-1]$. This follows from (\ref{eq:chainsum}). We can organize these coefficients into a triangle whose entries are recursively defined as follows. The $n$th row has $n$ entries, and the leftmost entry is always a $1$. The remaining entries are found by summing all numbers in the previous row that are above and to the left. Note this is equivalent to summing the two numbers directly to the left and directly above. This triangle is well-known and often referred to as the ``Catalan triangle''. See Figure~\ref{fig:cattri}.

\begin{figure}[ht]
\begin{tabular}{c||ccccccccc}
$n \backslash k$& $0$ & $1$ & 2 & 3 & 4 & 5 & 6\\
\hline
\hline
    1& 1 &  \\
    2& 1 & 1\\
    3& 1 & 2 & 2\\
    4& 1 & 3 & 5 & 5\\
    5& 1 & 4 & 9 & 14 & 14\\
    6& 1 & 5 & 14 & 28 & 42& 42 \\
    7& 1 & 6 & 20 & 48 & 90 & 132 & 132\\
    8& 1 & 7 & 27 & 75 & 165 & 297 & 429 & 429 
\end{tabular}
\caption{The Catalan triangle. The $(n,k)$ entry is denoted $t(n,k)$ and is the coefficient of $b(n,k)$ in $\#\Tr([n])$.}
\label{fig:cattri}
\end{figure}
Write $t(n,k)$ for the numbers in this triangle. Observe the Catalan numbers are already appearing in this table as the last entry in each row with $t(n,n-1)=\Cat(n-1)$. We outline the proof of this standard fact. First one can show the entries $t(n,k)$ have the closed form
\[
t(n,k) 
= \frac{n-k}{n}\binom{n-1+k}{k}
= \frac{n-k}{n}\binom{n-1+k}{n-1}
\]
by noting $\frac{n-0}{n}\binom{n-1+0}{0}=1$ for all $n$ and that the function $(n,k)\mapsto\frac{n-k}{n}\binom{n-1+k}{k}$ satisfies the defining recurrence of the triangle
\[
t(n,k)+t(n+1,k-1)=t(n+1, k).
\]
Now we can use the closed form for $t(n,k)$ and the hockey stick identity for binomial coefficients to prove that
\[
t(n,n-1)=\sum_{k=0}^{n-2} t(n-1,k)=\frac{1}{n}\binom{2n-2}{n-1}=\text{Cat}(n-1).
\]

Let's go back to the formula coming out of our recurrence. The numbers in the triangle are the coefficients of $b(n,k)$ in the sum to compute $b(0,0)=\#\Tr([n])$. Thus
\begin{equation}\label{eq:trnsum}
\#\Tr([n])=\sum_{k=0}^{n-1}t(n,k)b(n,k) = \sum_{k=0}^{n-1}(n-k+1)t(n,k).
\end{equation}
For example when $n=4$ this formula is
\[
\#\Tr([4])=5t(4,0)+4t(4,1)+3t(4,2)+2t(4,3),
\]
and indeed $42=5\cdot 1+4\cdot 3+ 3\cdot 5+2\cdot 5$, which matches $\#\Tr([4])=\Cat(5)$.

We use properties of the Catalan triangle to show the sum in~\ref{eq:trnsum} is equal to $t(n+2,n+1)$, which as we saw above is exactly $\Cat(n+1)$. Organize the terms in the sum $\sum_{k=0}^{n-1}(n-k+1)t(n,k)$ into an array where the first row has $(n+1)$ copies of $t(n,0)$, the next row has $n$ copies of $t(n,1)$, and so on until the final row has $2$ copies of $t(n,n-1)$. Notice if we sum each column of the array we get the values shown below:
\[\begin{array}{cccccc}
    t(n,0) &t(n,0) &t(n,0) & \cdots & t(n,0) & t(n,0)  \\
     & t(n,1) &t(n,1) & \cdots & t(n,1) & t(n,1)\\
     &  &t(n,2) & \cdots & t(n,2) & t(n,2)\\
     & & &\ddots & \vdots & \vdots\\
     & & & & t(n,n-1) & t(n,n-1)\\
     \hline
     t(n+1,0) & t(n+1,1) & t(n+1,2) & \cdots & t(n+1,n-1) & t(n+1,n).
\end{array}\]
Now summing along the bottom row gives us $t(n+2,n+1)=\text{Cat}(n+1)$, as expected.
\end{example}

\begin{example}\label{ex:longfork}
Let $\mathcal{F}_n$ denote the arborescence whose Hasse diagram is the ``long-fork of depth $n$''. This looks like a chain of $n$ nodes, where the last internal node has two children that are the two leaves. See $\mathcal{F}_4$ below, where the edges flow to the right:
    
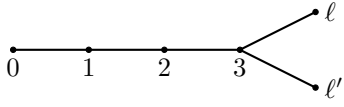
\begin{center}
        \begin{tikzpicture}
            \node (0) at (0,0) {$\vdot$};
            \node (1) at (1,0) {$\vdot$};
            \node (2) at (2,0) {$\vdot$};
            \node (3) at (3,0) {$\vdot$};
            \node (l1) at (4,0.5) {$\vdot$};
            \node (l2) at (4,-0.5) {$\vdot$};
            \draw (0) node[below]{$0$};
            \draw (1) node[below]{$1$};
            \draw (2) node[below]{$2$};
            \draw (3) node[below]{$3$};
            \draw (l1) node[right]{$\ell$};
            \draw (l2) node[right]{$\ell'$};
            \draw[thick] (0.center)--(1.center)--(2.center)--(3.center)--(l1.center);
            \draw[thick] (l2.center)--(3.center);
        \end{tikzpicture}
    \end{center}
We can use the work in our previous example to compute the number of transfer systems on $\mathcal{F}_n$. Note for all $i$, $b(\ell, i)=b(\ell', i)$ and both are equal to $b(n,i)$ in $[n]$. Also $n-1$ has two covers in $\mathcal{F}_n$, instead of just one cover as it did in $[n]$. Thus the formula for $b(\bot, 0)$ in $\mathcal{F}_n$ is given by replacing each $b(n,i)$ with $b(n,i)^2$ in the formula for $[n]$. We have
    \[
    \#\Tr(\mathcal{F}_n)=\sum_{k=0}^{n-1}t(n,k)b(n,k)^2 = \sum_{k=0}^{n-1} \frac{(n-k)(n-k+1)^2}{n}\binom{n-1+k}{n-1}.
    \]
With some rewriting and clever use of identities, one can show
    \[
      \#\Tr(\mathcal{F}_n)=\sum_{k=0}^{n-1} \frac{(n-k)(n-k+1)^2}{n}\binom{n-1+k}{n-1} = \frac{2(5n+3)}{(n+2)(n+3)}\binom{2n+1}{n}.
    \]
Below are the first handful of numbers in this sequence, starting with $n=1$:
    \[4,~ 13,~ 42,~ 138,~ 462,~ 1573, ~5424, ~19006.\]
This formula also follows from the bijection given in Theorem~\ref{thm:bijcoalhis} and the result in \cite[Corollary 3.9]{rosenberg_counting_2007} for the number of coalescent histories on the ``pseudocaterpillar tree with $n+3$ taxa $Y_{n+3}$'' (here the observation is that $\mathcal{F}_n$ is obtained by removing all leaves from $Y_{n+3}$).
\end{example}

\begin{example}\label{ex:binary}
Let $\mathcal{T}_m$ be a full binary tree of depth $m$, so a tree where each vertex of rank less than $m$ has exactly two children and the leaves are $m$ edges from the root. Below are $\T_0$, $\T_1$, $\T_2$, and $\T_3$:
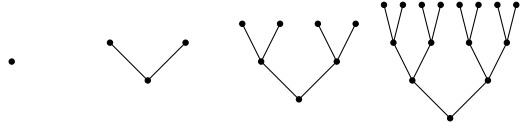
\begin{center}
\begin{tikzpicture}
    \node (0) at (0.2,0) {\begin{tikzpicture}[scale=0.5]
    \node (0) at (0,0) {$\vdot$};
    \end{tikzpicture}};
    \node (1) at (2,0) {\begin{tikzpicture}[scale=0.5]
    \node (0) at (0,0) {$\vdot$};
    \node (11) at (-1,1) {$\vdot$};
    \node (12) at (1,1) {$\vdot$};
    \draw(0.center)--(11.center);
    \draw (0.center)--(12.center);
    \end{tikzpicture}};
    \node (2) at (4,0) {\begin{tikzpicture}[scale=0.5]
    \node (0) at (0,0) {$\vdot$};
    \node (11) at (-1,1) {$\vdot$};
    \node (12) at (1,1) {$\vdot$};
    \node (21) at (-1.5,2) {$\vdot$};
    \node (22) at (-0.5,2) {$\vdot$};
    \node (23) at (1.5,2) {$\vdot$};
    \node (24) at (0.5,2) {$\vdot$};
    \draw(0.center)--(11.center)--(21.center);
    \draw (11.center)--(22.center);
    \draw (0.center)--(12.center)--(23.center);
    \draw (12.center)--(24.center);
    \end{tikzpicture}};
    \node (3) at (6,0) {\begin{tikzpicture}[scale=0.5]
    \node (0) at (0,0) {$\vdot$};
    \node (11) at (-1,1) {$\vdot$};
    \node (12) at (1,1) {$\vdot$};
    \node (21) at (-1.5,2) {$\vdot$};
    \node (22) at (-0.5,2) {$\vdot$};
    \node (24) at (1.5,2) {$\vdot$};
    \node (23) at (0.5,2) {$\vdot$};
    \node (31) at (-1.75,3) {$\vdot$};
    \node (32) at (-1.25,3) {$\vdot$};
    \node (33) at (-.75,3) {$\vdot$};
    \node (34) at (-.25,3) {$\vdot$};
    \node (38) at (1.75,3) {$\vdot$};
    \node (37) at (1.25,3) {$\vdot$};
    \node (36) at (.75,3) {$\vdot$};
    \node (35) at (.25,3) {$\vdot$};
    \draw(0.center)--(11.center)--(21.center)--(31.center);
    \draw (21.center)--(32.center);
    \draw (11.center)--(22.center)--(33.center);
    \draw (22.center)--(34.center);
    \draw (0.center)--(12.center)--(23.center)--(35.center);
    \draw (23.center)--(36.center);
    \draw (12.center)--(24.center)--(37.center);
    \draw (24.center)--(38.center);
    \end{tikzpicture}};
\end{tikzpicture}
\end{center}
Observe if $x$, $y$ are two vertices in $\T_m$ with $\rho(x)=\rho(y)$, then $b(x,n)= b(y,n)$ for all $n$. Thus we write $\ell_j$ for a generic vertex in $B_m$ of rank $j$. Our recurrence gives:
\begin{alignat*}{3}
&m=0:~~b(\bot,0)&&=b(\ell_0,0),\\
&m=1:~~b(\bot,0)&&=b(\ell_0,0)=(b(\ell_1,0))^2,\\
&m=2:~~b(\bot,0)&&=(b(\ell_1,0))^2=(b(\ell_2,0)^2+b(\ell_2,1)^2)^2\\
&m=3:~~b(\bot,0)&&=(b(\ell_2,0)^2+b(\ell_2,1)^2)^2\\
& &&=[(b(\ell_3,0)^2+b(\ell_3,1)^2+b(\ell_3,2)^2)^2+(b(\ell_3,1)^2+b(\ell_3,2)^2)^2]^2
\end{alignat*}
Note the pattern: to get the formula for $\T_m$ from the formula for $\T_{m-1}$, replace every term of the form $b(\ell_{m-1}, k)$ with $\sum_{j=k}^{m-1}b(\ell_m,j)^2$ (this holds from the recurrence and that $\ell_{m-1}$ has two covers). Using our formula in Lemma~\ref{lem:leafcount} we get:
\begin{align*}
\#\Tr(\T_0)&=1\\
\#\Tr(\T_1)&= 2^2,\\
\#\Tr(\T_2)&=(3^2+2^2)^2,\\ 
\#\Tr(\T_3)&=((4^2+3^2+2^2)^2+ (3^2+2^2)^2)^2,\\  \#\Tr(\T_4)&=([(5^2+4^2+3^2+2^2)^2+(4^2+3^2+2^2)^2 + (3^2+2^2)^2]^2+ \\
&\quad \quad[(4^2+3^2+2^2)^2+ (3^2+2^2)^2]^2)^2.
\end{align*}
In order to get from $\#\Tr(\T_m)$ from $\#\Tr(\T_{m-1})$, replace each term of the form $k^2$ by $\sum_{j=2}^{k+1} j^2$. This sequence grows very quickly, as seen by just considering the first five terms,
\[
1,~~ 4, ~~169,~~ 1020100, ~~270062420147776.
\]
We do not know of a closed form for this sequence. This sequence can be found in the Online Enyclopedia of Integer Sequences (OEIS) as entry A306402 \cite{oeisbintree}. This also counts the number of coalescent histories for full binary trees. See Appendix~\ref{ap:coalhis} and Theorem~\ref{thm:bijcoalhis} for connections between transfer systems and coalescent histories.
\end{example}

\section{Opposite sets through weak factorization systems}\label{sec:op}
Our convention thus far has been to partially order the vertex set of a rooted tree so that the root is the unique minimal element and the leaves are maximal elements. What happens if we reverse the ordering and try to enumerate transfer systems on the resulting ``upside down tree''?  To answer this, we more generally explore the relationship between the transfer system lattice of a poset and the transfer system lattice of its dual or opposite poset. 

In \cite[Theorem 4.21]{selfdual} the authors show that given any self-dual poset $\P$ (i.e. any poset $\P$ such that $\P^{op}\cong \P$), the lattice $\Tr(\P)$ is also self-dual. The proof uses constructions from category theory, and along the way they establish a bijection between transfer systems and weak factorization systems (see Definition~\ref{def:wfs}) on a lattice. In their paper they use version $(2')$ of the restriction condition as mentioned in Remark~\ref{rem:difres} (which is equivalent to $(2)$ for lattices). We use slight modifications of their argument now with version $(2)$ of the restriction condition to prove a bijection between transfer systems and weak factorization systems for any finite poset $\P$, as was suggested in \cite[Remark 2.4]{BMO2}. We then use this bijection to give an order reversing bijection from $\Tr(\P)$ to $\Tr(\P^{op})$. 

\begin{notation}
Let $R$ be a relation on a poset $\P$ that refines $\leq$. We will adopt the notation $x\to y$ whenever $x\leq y$, and we'll write $x\to y \in R$ whenever $x~R~y$. This notation is borrowed from category theory: given a poset $\P$, we can view $\P$ as a category where the elements of $\P$ are objects and the relations in $\P$ give morphisms. Then the relation $R$ corresponds to a subset of morphisms. 

We will write $R^{\op}$ for the relation on $\P^{\op}$ defined by 
\[
\Rop\coloneqq\{y \to x \mid x \to y \in R\}.
\]
Given a poset $(\P, \leq)$ we'll write $\P^{op}$ for the opposite poset $(\P, \leq^{op})$.
\end{notation}

Roughly speaking, a weak factorization system on a category is a collection of ``right'' and of ``left'' morphisms that can be used to factor any arrow in the category. Furthermore, these collections have nice lifting properties with respect to each other. We give the precise definition below. For a more thorough treatment on weak factorization systems and how they are used in homotopy theory see, for example, \cite[Section 14.1]{algtopbook}.

\begin{defn}
     Let $\P$ be a poset with $x,y\in \P$. We say that $x \to y$ has the \textbf{right lifting property} (rlp) with respect to $a \to b$ if $a \leq x$ and $b \leq y$ implies $b \leq x$. That is, if given the square below, we can find a ``lift'' as shown with the dashed line:
     \begin{center}
            \begin{tikzcd}[ampersand replacement=\&]
                a \arrow[r] 
                \& x \arrow[d] \\
                b \arrow[ur, dashrightarrow] \arrow[u, leftarrow]
                \& y \arrow[l, leftarrow]
            \end{tikzcd}
    \end{center}
In this case we say $a\to b$ has the \textbf{left-lifting property} with respect to $x\to y$.
    
\noindent Given a relation $R$ on $\P$ that refines $\leq$ define the \textbf{right lifting set} $R^{\lift}$ as the relation \[R^{\lift} = \{x \to y \mid x \to y \text{ has the rlp with respect to all } a \to b \in R\}.\]
We can similarly define the \textbf{left lifting set} $\llR$.
\end{defn}

\begin{defn}\label{def:wfs}
Let $\mathcal{C}$ be a category and suppose $R$ and $L$ are collections of morphisms from $\mathcal{C}$. We say $(L,R)$ is a \textbf{weak factorization system} on $\mathcal{C}$ if
\begin{enumerate}[leftmargin=*]
\item[(i)] for every morphism $f$ in $\mathcal{C}$, there exists $i\in L$ and $p\in R$ so that $f$ factors as $f=p\circ i$, and
\item[(ii)] $R=L^{\lift}$ and $L=\prescript{\lift}{}{R}$.
\end{enumerate}
We write $\wfs(\mathcal{C})$ for the collection of all weak factorization systems on $\mathcal{C}$.
\end{defn}

Let $\P$ be a finite poset. We can give $\wfs(\P)$ a partial order by defining $(L,R)\leq (L',R')$ if and only if $R\subseteq R'$. We first note the following relationship between weak factorization systems on $\P$ and $\P^{\op}$. 

\begin{prop}
Let $\P$ be a finite poset. Then we have an order-reversing bijection $\wfs(\P)\to \wfs(\P^{\op})$ given by $(L,R)\mapsto (R^{\op}, L^{\op})$.
\end{prop}

The proof readily follows from the definitions. Our next goal is to prove we have a map $\Tr(\P)\to \wfs(\P)$ defined by $R\mapsto (\prescript{\lift}{}{R}, R)$ and that this gives an isomorphism of lattices. With the proposition above, we will get that $\Tr(\P)^{\op}\cong \Tr(\P^{\op})$ as an easy corollary.

We start by proving two lemmas. The first was already shown to hold for lattices in \cite[Proposition 4.11]{selfdual}. Our proof closely follows theirs, except we need to be careful about using pseudo-meets instead of meets.

\begin{lemma} \label{lem:factoring}
Given a finite poset $\P$ and $R \in \Tr(\P)$, we can factor any $x \to y \in \P$ as $x \to x' \to y$ where $x \to x' \in \llR$ and $x' \to y \in R$.
\end{lemma}
\begin{proof}
Suppose first $x \to y \in \llR$. Note $y \to y \in R$  because $R$ contains all reflexive relations, and thus we can just use the factorization of $x \to y$ as $x \to y \to y$.

Now suppose $x \to y \notin \llR$. Then there exists $z \to w \in R$ such that $x \to y$ does not have the left lifting property with respect to $z \to w$. This is illustrated in the diagram
    \begin{center}
        \begin{tikzcd}
            x \arrow[r] 
            & z \arrow[d, "R"] \\
            y \arrow[r] \arrow[ur, "\times" description, rightsquigarrow] \arrow[u, leftarrow]
            & w.
        \end{tikzcd}
    \end{center}
Hence $y \nleq z$. Since $\P$ is not necessarily a lattice, the meet of $y$ and $z$ need not exist. Instead fix $m$ to be a maximal element in the set \[\{p\in \P \mid p \leq y,~ p \leq z, \text{ and } x\leq p\}.\] Note $x$ is an element of this finite set, so we know such a maximal element exists. If the meet of $y$ and $z$ does exist, then we just get $m=y\meet z$. Otherwise $m$ is a pseudo-meet. 
    
By the restriction condition of $R$ we have $m~ R~y$. By construction of $m$ we can factor $x \to y$ as $x \to m \to y$. See the diagram below:
    \begin{center}
        \begin{tikzcd}
            x \arrow[drr, bend left] \arrow[ddr, bend right] \arrow[dr] & & \\[-1.5em] 
            & m \arrow[r]
            & z \arrow[d, "R"] \\
            & y \arrow[r] \arrow[u, "R", leftarrow]
            & w.
        \end{tikzcd}
    \end{center}
Now if $x \to m \in \llR$, then we have factored $x \to y$ as desired. Otherwise, we can repeat this process with $x \to m$ in place of $x \to y$. Note $R$ is transitive, so a factorization of $x\to m$ will also yield a factorization of $x\to y$.
    
Observe $y\nleq z$ implies $m\neq y$, so it must be $m < y$ in $P$. By repeating the process above, we will eventually have a desired factoring of $x\to y$ or we will reach a factoring $x\to b\overset{R}{\to} y$ where $b$ is a minimal element. In this case $x \to b$ implies $x = b$ so $x \to b$ is a reflexive relation, which is in $\llR$.
\end{proof}

\begin{remark}
Observe if $\P$ is a poset with elements that have pseudo-meet(s) but no honest meet, then the above proof only works if we use the restriction condition $(2)$ from Definition~\ref{def:transf} as opposed to $(2')$ from Remark~\ref{rem:difres}. 
\end{remark}

\begin{lemma} \label{lem:leftrightlift}
Let $\P$ be a finite poset. Given $R \in \Tr(\P)$, we have that $\lRl = R$.
\end{lemma}
\begin{proof}
First note that $R \subseteq \lRl$ directly from the definition of left and right lifting sets. To show that $\lRl \subseteq R$, fix $x \to y \in \lRl$. By Lemma \ref{lem:factoring}, there exists $x' \in \P$ so that we can factor $x \to y$ as $x \to x' \to y$ where $x \to x' \in \llR$ and $x' \to y \in R$. Consider the following commutative square:
    \begin{center}
        \begin{tikzcd}
            x \arrow[r] 
            & x \arrow[d, "\lRl"] \\
            x' \arrow[ur, dashrightarrow] \arrow[u, "\llR", leftarrow]
            & y. \arrow[l, "R", leftarrow]
        \end{tikzcd}
    \end{center}
Since $x \to y \in \lRl$, we know there must exist a lift $x' \to x$ in $\P$. So we have $x \leq x'$ and $x' \leq x$ which implies $x = x'$. Thus since $x'~R~y$ we have that $x~R~y$. Hence $\lRl \subseteq R$, so $\lRl = R$.
\end{proof}

The two lemmas imply that we have a function $\Tr(\P)\to \wfs(\P)$, $R\mapsto \llR$. We next show there is an inverse function $\wfs(\P)\to\Tr(\P)$ defined by $(L,R)\mapsto R$.

\begin{lemma}\label{lem:wfstransf}
Suppose $(L,R)$ is a weak factorization system on $\P$. Then $R$ is a transfer system on $\P$.
\end{lemma}
\begin{proof}
First note $R$ refines $\leq$ by definition. We also have $R=L^{\lift}$, and it is straightforward to check any right-lifting set contains all identity morphisms and is closed under composition (that is, $R$ is reflexive and transitive). We only need to verify the restriction condition from Definition~\ref{def:transf}.

Suppose $x\to z\in R$ and $y\leq z$. Let $m$ be a pseudo-meet of $x$ and $y$. We must show $m\to y\in R$. By definition of weak factorization system we can factor $m\to y$ as $m\to w \to y$ where $m\to w\in L$ and $w\to y\in R$ (this implies $m\leq w \leq y$). Now since $x\to z\in R$, we get a lift $w\leq x$ as shown below
\begin{center}
        \begin{tikzcd}
            m\arrow[d,"L" left]\arrow[r]&x\arrow[d,"R"]\\
            w\arrow[ru,dashed]\arrow[r]&z.
        \end{tikzcd}
\end{center}
We thus have $w\leq x$, $w\leq y$, and $m\leq w$, and so $m=w$ by definition of pseudo-meet. In particular, $m\to y\in R$ as desired.
\end{proof}

\begin{thm}\label{thm:wfs}
Let $\P$ be a finite poset. Then the function $\Tr(\P)\to \wfs(\P)$ given by $R\mapsto (\prescript{\lift}{}{R}, R)$ is an isomorphism of posets. 
\end{thm}
\begin{proof}
It follows from Lemmas~\ref{lem:factoring} and \ref{lem:leftrightlift} that $(\prescript{\lift}{}{R}, R)$ is a weak factorization system on $\P$ for any transfer system $R$. Thus we have an order-preserving function $\Tr(\P)\to \wfs(\P)$ given by $R\mapsto (\prescript{\lift}{}{R}, R)$. By Lemma~\ref{lem:wfstransf} we also have an order-preserving function $\wfs(\P)\to \Tr(\P)$ given by $(L,R)\mapsto R$. These two functions are clearly inverses of one another. We conclude $\Tr(\P)\cong \wfs(\P)$.
\end{proof}

Returning to our question at the start of the section, we now get the promised isomorphism between $\Tr(\P)^{\op}$ and $\Tr(\P^{\op})$.

\begin{cor}\label{cor:oppable}
For any finite poset $\P$ we have that $\Tr(\P^{\op})\cong \Tr(\P)^{\op}$ as lattices. Explicitly, we have an order-reversing bijection $\gamma_\P\colon \Tr(\P)\to \Tr(\P^{\op})$ defined by $\gamma_\P(R)=(\prescript{\lift}{}{R})^{\op}=(R^{\lift})^{\op}$. In particular, $\#\Tr(\P)=\#\Tr(\P^{\op})$.
\end{cor}

\begin{remark}\label{rem:difresag}
Recall the two versions of the restriction condition discussed in Remark~\ref{rem:difres}. As mentioned, if we take version $(2')$ instead of $(2)$ then $\Tr(\P)\neq \# \Tr(\P^{\op})$ where $\P$ is the poset whose Hasse diagram is shown below:
    
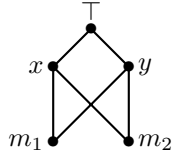
\begin{center}
    \begin{tikzpicture}
    \node (a) at (0,0) {$\bullet$};
    \node (b) at (1,0) {$\bullet$};
    \node (c) at (0,1) {$\bullet$};
    \node (d) at (1,1) {$\bullet$};
    \node (e) at (0.5,1.5) {$\bullet$};
    \draw[thick] (d.center)--(a.center)--(c.center)--(e.center)--(d.center)--(b.center)--(c.center);
    \draw (a) node[left]{$m_1$};
    \draw (b) node[right]{$m_2$};
    \draw (c) node[left]{$x$};
    \draw (d) node[right]{$y$};
    \draw (e) node[above]{$\top$};
    \end{tikzpicture}  
    \end{center}
For this poset $\P$ we also do not have a function $\Tr(\P)\to \wfs(\P)$ defined by $R\mapsto (\llR, R)$ if we use $(2')$ as the restriction condition. For example, take $R$ to be the transfer system (according to $(2')$) whose only non-reflexive relation is $x\to \top$. Then one can verify $\lRl$ contains both $m_1\to y$ and $m_2\to y$, so $R$ cannot be the right-lifting set in a weak factorization system because $\lRl\neq R$. 
\end{remark}

In general, the function $\gamma_P$ in Corollary~\ref{cor:oppable} can be somewhat complicated to compute. But we record the following property of $\gamma_\P$. Roughly speaking, this states that, ignoring reflexive relations, the graph of $\gamma_\P(R)$ is contained in the ``flipped complement'' of the graph of $R$, i.e. the non-reflexive relations in $\gamma_\P(R)$ are contained in $(R^{op})^c$.
\begin{prop}\label{prop:flipcomp}
    Let $\P$ be a poset and $R$ be a transfer system on $\P$. If $x,y$ are distinct elements in $\P$ with $y\to x \in \gamma_\P(R)$, then $x\to y \not \in R$. 
\end{prop}
\begin{proof}
Let $x,y$ be distinct elements so that $y\to x\in \gamma_P(R)$, and suppose contrary $x\to y\in R$. Then since $y\to x\in R^{op}$, we'd have the following lifting diagram
    \begin{center}
            \begin{tikzcd}[ampersand replacement=\&]
                y \arrow[r] \arrow[d,"R^{op}" left]
                \& y \arrow[d,"\gamma_P(R)=(R^{\op})^{\lift}"] \\
                x \arrow[ur, dashrightarrow]
                \& x. \arrow[l, leftarrow]
            \end{tikzcd}
    \end{center}
But this implies $x=y$, which is a contradiction.
\end{proof}
There are many examples where $\gamma_P$ maps a transfer system $R$ exactly to its flipped complement, but this is not always the case. See Example~\ref{ex:yop} for both examples and non-examples.

\begin{example}\label{ex:yop}
Let $Y$ denote the four-element poset whose Hasse diagram looks like the letter $Y$.
Figure~\ref{fig:yop} shows the lattice of transfer systems on $Y$ on the left and the lattice of transfer systems on the upside down $Y$ ($Y^{\op}$) on the right. Note $\Tr(Y^{\op})$ is drawn with the opposite orientation, so we have instead labeled the column $\Tr(Y^{\op})^{\op}$. The lattices are arranged so that the image under $\gamma_Y$ of a transfer system on the left is shown in the corresponding spot on the right. The two blue transfer systems are the only ones not mapped to their flipped compliment, ignoring the reflexive relations (see Proposition~\ref{prop:flipcomp}).

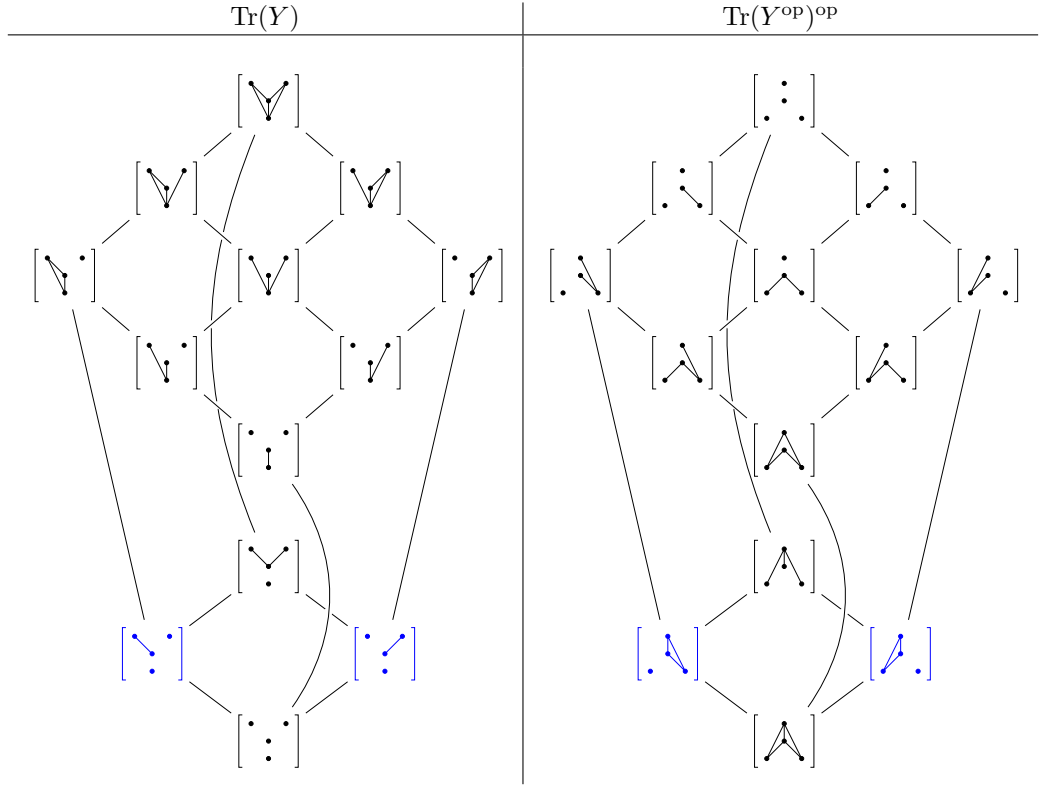
\begin{figure}[ht]
\begin{tabular}{c|c}
$\Tr(Y)$ & $\Tr(Y^{\op})^{\op}$\\
\hline
\\
\scalebox{0.77}{
\begin{tikzpicture}
    \node (1) at (0,1){
    \begin{tikzpicture}[scale=0.3]
        \node (a1) at (-1,2) {\vdot};
        \node (a2) at (1,2) {\vdot};
        \node (b) at (0,1) {\vdot};
        \node (c) at (0,0) {\vdot};
        \draw (-1.5,2.5)--(-1.7,2.5)--(-1.7,-0.5)--(-1.5,-0.5);
        \draw (1.5,2.5)--(1.7,2.5)--(1.7,-0.5)--(1.5,-0.5);
    \end{tikzpicture}
    };
    \node (2) at (-2,2.5){
    \begin{tikzpicture}[scale=0.3]
        \node[blue] (a1) at (-1,2) {\vdot};
        \node[blue] (a2) at (1,2) {\vdot};
        \node[blue] (b) at (0,1) {\vdot};
        \node[blue] (c) at (0,0) {\vdot};
        \draw[blue] (b.center)--(a1.center);
        \draw[blue] (-1.5,2.5)--(-1.7,2.5)--(-1.7,-0.5)--(-1.5,-0.5);
        \draw[blue] (1.5,2.5)--(1.7,2.5)--(1.7,-0.5)--(1.5,-0.5);
    \end{tikzpicture}
    };
    \node (3) at (2,2.5){
    \begin{tikzpicture}[scale=0.3]
        \node[blue] (a1) at (-1,2) {\vdot};
        \node[blue] (a2) at (1,2) {\vdot};
        \node[blue] (b) at (0,1) {\vdot};
        \node[blue] (c) at (0,0) {\vdot};
        \draw[blue] (b.center)--(a2.center);
        \draw[blue] (-1.5,2.5)--(-1.7,2.5)--(-1.7,-0.5)--(-1.5,-0.5);
        \draw[blue] (1.5,2.5)--(1.7,2.5)--(1.7,-0.5)--(1.5,-0.5);
    \end{tikzpicture}
    };
    \node (4) at (0,4){
    \begin{tikzpicture}[scale=0.3]
        \node (a1) at (-1,2) {\vdot};
        \node (a2) at (1,2) {\vdot};
        \node (b) at (0,1) {\vdot};
        \node (c) at (0,0) {\vdot};
        \draw (b.center)--(a1.center);
        \draw (b.center)--(a2.center);
        \draw (-1.5,2.5)--(-1.7,2.5)--(-1.7,-0.5)--(-1.5,-0.5);
        \draw (1.5,2.5)--(1.7,2.5)--(1.7,-0.5)--(1.5,-0.5);
    \end{tikzpicture}
    };
    \node (5) at (0,6){
    \begin{tikzpicture}[scale=0.3]
        \node (a1) at (-1,2) {\vdot};
        \node (a2) at (1,2) {\vdot};
        \node (b) at (0,1) {\vdot};
        \node (c) at (0,0) {\vdot};
        \draw (c.center)--(b.center);
        \draw (-1.5,2.5)--(-1.7,2.5)--(-1.7,-0.5)--(-1.5,-0.5);
        \draw (1.5,2.5)--(1.7,2.5)--(1.7,-0.5)--(1.5,-0.5);
    \end{tikzpicture}
    };
    \node (6) at (-1.75,7.5){
    \begin{tikzpicture}[scale=0.3]
        \node (a1) at (-1,2) {\vdot};
        \node (a2) at (1,2) {\vdot};
        \node (b) at (0,1) {\vdot};
        \node (c) at (0,0) {\vdot};
        \draw (c.center)--(b.center);
        \draw (c.center)--(a1.center);
        \draw (-1.5,2.5)--(-1.7,2.5)--(-1.7,-0.5)--(-1.5,-0.5);
        \draw (1.5,2.5)--(1.7,2.5)--(1.7,-0.5)--(1.5,-0.5);
    \end{tikzpicture}
    };
    \node (7) at (1.75,7.5){
    \begin{tikzpicture}[scale=0.3]
        \node (a1) at (-1,2) {\vdot};
        \node (a2) at (1,2) {\vdot};
        \node (b) at (0,1) {\vdot};
        \node (c) at (0,0) {\vdot};
        \draw (c.center)--(b.center);
        \draw (c.center)--(a2.center);
        \draw (-1.5,2.5)--(-1.7,2.5)--(-1.7,-0.5)--(-1.5,-0.5);
        \draw (1.5,2.5)--(1.7,2.5)--(1.7,-0.5)--(1.5,-0.5);
    \end{tikzpicture}
    };
    \node (8) at (-3.5,9){
    \begin{tikzpicture}[scale=0.3]
        \node (a1) at (-1,2) {\vdot};
        \node (a2) at (1,2) {\vdot};
        \node (b) at (0,1) {\vdot};
        \node (c) at (0,0) {\vdot};
        \draw (c.center)--(b.center)--(a1.center);
        \draw (c.center)--(a1.center);
        \draw (-1.5,2.5)--(-1.7,2.5)--(-1.7,-0.5)--(-1.5,-0.5);
        \draw (1.5,2.5)--(1.7,2.5)--(1.7,-0.5)--(1.5,-0.5);
    \end{tikzpicture}
    };
    \node (9) at (0,9){
    \begin{tikzpicture}[scale=0.3]
        \node (a1) at (-1,2) {\vdot};
        \node (a2) at (1,2) {\vdot};
        \node (b) at (0,1) {\vdot};
        \node (c) at (0,0) {\vdot};
        \draw (c.center)--(b.center);
        \draw (c.center)--(a1.center);
        \draw (c.center)--(a2.center);
        \draw (-1.5,2.5)--(-1.7,2.5)--(-1.7,-0.5)--(-1.5,-0.5);
        \draw (1.5,2.5)--(1.7,2.5)--(1.7,-0.5)--(1.5,-0.5);
    \end{tikzpicture}
    };
    \node (10) at (3.5,9){
    \begin{tikzpicture}[scale=0.3]
        \node (a1) at (-1,2) {\vdot};
        \node (a2) at (1,2) {\vdot};
        \node (b) at (0,1) {\vdot};
        \node (c) at (0,0) {\vdot};
        \draw (c.center)--(b.center)--(a2.center);
        \draw (c.center)--(a2.center);
        \draw (-1.5,2.5)--(-1.7,2.5)--(-1.7,-0.5)--(-1.5,-0.5);
        \draw (1.5,2.5)--(1.7,2.5)--(1.7,-0.5)--(1.5,-0.5);
    \end{tikzpicture}
    };
    \node (11) at (-1.75,10.5){
    \begin{tikzpicture}[scale=0.3]
        \node (a1) at (-1,2) {\vdot};
        \node (a2) at (1,2) {\vdot};
        \node (b) at (0,1) {\vdot};
        \node (c) at (0,0) {\vdot};
        \draw (c.center)--(b.center)--(a1.center);
        \draw (c.center)--(a1.center);
        \draw (c.center)--(a2.center);
        \draw (-1.5,2.5)--(-1.7,2.5)--(-1.7,-0.5)--(-1.5,-0.5);
        \draw (1.5,2.5)--(1.7,2.5)--(1.7,-0.5)--(1.5,-0.5);
    \end{tikzpicture}
    };
    \node (12) at (1.75,10.5){
    \begin{tikzpicture}[scale=0.3]
        \node (a1) at (-1,2) {\vdot};
        \node (a2) at (1,2) {\vdot};
        \node (b) at (0,1) {\vdot};
        \node (c) at (0,0) {\vdot};
        \draw (c.center)--(b.center)--(a2.center);
        \draw (c.center)--(a2.center);
        \draw (c.center)--(a1.center);
        \draw (-1.5,2.5)--(-1.7,2.5)--(-1.7,-0.5)--(-1.5,-0.5);
        \draw (1.5,2.5)--(1.7,2.5)--(1.7,-0.5)--(1.5,-0.5);
    \end{tikzpicture}
    };
    \node (13) at (0,12){
    \begin{tikzpicture}[scale=0.3]
        \node (a1) at (-1,2) {\vdot};
        \node (a2) at (1,2) {\vdot};
        \node (b) at (0,1) {\vdot};
        \node (c) at (0,0) {\vdot};
        \draw (c.center)--(b.center)--(a2.center);
        \draw (c.center)--(a2.center);
        \draw (c.center)--(a1.center)--(b.center);
        \draw (-1.5,2.5)--(-1.7,2.5)--(-1.7,-0.5)--(-1.5,-0.5);
        \draw (1.5,2.5)--(1.7,2.5)--(1.7,-0.5)--(1.5,-0.5);
    \end{tikzpicture}
    };
    \draw (4) to [bend left=22] (13);
    \filldraw[white] (-0.88,6.77) circle (1.7pt);
    \filldraw[white] (-0.79,9.68) circle (1.7pt);
    \filldraw[white] (-0.98,8.16) circle (1.7pt);
    \draw (1)--(2)--(4)--(3)--(1);
    \draw (5)--(6)--(9)--(7)--(5);
    \draw (6)--(8)--(11)--(9)--(12)--(13)--(11);
    \draw (7)--(10)--(12);
    \filldraw[white] (1.02,3.21) circle (1.9pt); 
    \draw (1) to [bend right=35] (5);
    \draw (2)--(8);
    \draw (3)--(10);
\end{tikzpicture}}&
\scalebox{0.77}{
\begin{tikzpicture}
    \node (1) at (0,1){
    \begin{tikzpicture}[scale=0.3]
        \node (a1) at (-1,-2) {\vdot};
        \node (a2) at (1,-2) {\vdot};
        \node (b) at (0,-1) {\vdot};
        \node (c) at (0,0) {\vdot};
        \draw (-1.5,-2.5)--(-1.7,-2.5)--(-1.7,0.5)--(-1.5,0.5);
        \draw (1.5,-2.5)--(1.7,-2.5)--(1.7,0.5)--(1.5,0.5);
        \draw (a1.center)--(b.center)--(c.center)--(a1.center);
        \draw (c.center)--(a2.center)--(b.center);
    \end{tikzpicture}
    };
    \node (2) at (-2,2.5){
    \begin{tikzpicture}[scale=0.3]
        \node[blue] (a1) at (-1,-2) {\vdot};
        \node[blue] (a2) at (1,-2) {\vdot};
        \node[blue] (b) at (0,-1) {\vdot};
        \node[blue] (c) at (0,0) {\vdot};
        \draw[blue] (-1.5,-2.5)--(-1.7,-2.5)--(-1.7,0.5)--(-1.5,0.5);
        \draw[blue] (1.5,-2.5)--(1.7,-2.5)--(1.7,0.5)--(1.5,0.5);
        \draw[blue] (c.center)--(a2.center)--(b.center)--(c.center);
    \end{tikzpicture}
    };
    \node (3) at (2,2.5){
    \begin{tikzpicture}[scale=0.3]
        \node[blue] (a1) at (-1,-2) {\vdot};
        \node[blue] (a2) at (1,-2) {\vdot};
        \node[blue] (b) at (0,-1) {\vdot};
        \node[blue] (c) at (0,0) {\vdot};
        \draw[blue] (-1.5,-2.5)--(-1.7,-2.5)--(-1.7,0.5)--(-1.5,0.5);
        \draw[blue] (1.5,-2.5)--(1.7,-2.5)--(1.7,0.5)--(1.5,0.5);
        \draw[blue] (c.center)--(a1.center)--(b.center)--(c.center);
    \end{tikzpicture}
    };
    \node (4) at (0,4){
    \begin{tikzpicture}[scale=0.3]
        \node (a1) at (-1,-2) {\vdot};
        \node (a2) at (1,-2) {\vdot};
        \node (b) at (0,-1) {\vdot};
        \node (c) at (0,0) {\vdot};
        \draw (-1.5,-2.5)--(-1.7,-2.5)--(-1.7,0.5)--(-1.5,0.5);
        \draw (1.5,-2.5)--(1.7,-2.5)--(1.7,0.5)--(1.5,0.5);
        \draw (b.center)--(c.center)--(a1.center);
        \draw (c.center)--(a2.center);
    \end{tikzpicture}
    };
    \node (5) at (0,6){
    \begin{tikzpicture}[scale=0.3]
        \node (a1) at (-1,-2) {\vdot};
        \node (a2) at (1,-2) {\vdot};
        \node (b) at (0,-1) {\vdot};
        \node (c) at (0,0) {\vdot};
        \draw (-1.5,-2.5)--(-1.7,-2.5)--(-1.7,0.5)--(-1.5,0.5);
        \draw (1.5,-2.5)--(1.7,-2.5)--(1.7,0.5)--(1.5,0.5);
        \draw (a1.center)--(b.center)--(a2.center)--(c.center)--(a1.center);
    \end{tikzpicture}
    };
    \node (6) at (-1.75,7.5){
    \begin{tikzpicture}[scale=0.3]
        \node (a1) at (-1,-2) {\vdot};
        \node (a2) at (1,-2) {\vdot};
        \node (b) at (0,-1) {\vdot};
        \node (c) at (0,0) {\vdot};
        \draw (-1.5,-2.5)--(-1.7,-2.5)--(-1.7,0.5)--(-1.5,0.5);
        \draw (1.5,-2.5)--(1.7,-2.5)--(1.7,0.5)--(1.5,0.5);
        \draw (a1.center)--(b.center)--(a2.center)--(c.center);
    \end{tikzpicture}
    };
    \node (7) at (1.75,7.5){
    \begin{tikzpicture}[scale=0.3]
        \node (a1) at (-1,-2) {\vdot};
        \node (a2) at (1,-2) {\vdot};
        \node (b) at (0,-1) {\vdot};
        \node (c) at (0,0) {\vdot};
        \draw (-1.5,-2.5)--(-1.7,-2.5)--(-1.7,0.5)--(-1.5,0.5);
        \draw (1.5,-2.5)--(1.7,-2.5)--(1.7,0.5)--(1.5,0.5);
        \draw (a2.center)--(b.center)--(a1.center)--(c.center);
    \end{tikzpicture}
    };
    \node (8) at (-3.5,9){
    \begin{tikzpicture}[scale=0.3]
        \node (a1) at (-1,-2) {\vdot};
        \node (a2) at (1,-2) {\vdot};
        \node (b) at (0,-1) {\vdot};
        \node (c) at (0,0) {\vdot};
        \draw (-1.5,-2.5)--(-1.7,-2.5)--(-1.7,0.5)--(-1.5,0.5);
        \draw (1.5,-2.5)--(1.7,-2.5)--(1.7,0.5)--(1.5,0.5);
        \draw (b.center)--(a2.center)--(c.center);
    \end{tikzpicture}
    };
    \node (9) at (0,9){
    \begin{tikzpicture}[scale=0.3]
        \node (a1) at (-1,-2) {\vdot};
        \node (a2) at (1,-2) {\vdot};
        \node (b) at (0,-1) {\vdot};
        \node (c) at (0,0) {\vdot};
        \draw (-1.5,-2.5)--(-1.7,-2.5)--(-1.7,0.5)--(-1.5,0.5);
        \draw (1.5,-2.5)--(1.7,-2.5)--(1.7,0.5)--(1.5,0.5);
        \draw (a1.center)--(b.center)--(a2.center);
    \end{tikzpicture}
    };
    \node (10) at (3.5,9){
    \begin{tikzpicture}[scale=0.3]
        \node (a1) at (-1,-2) {\vdot};
        \node (a2) at (1,-2) {\vdot};
        \node (b) at (0,-1) {\vdot};
        \node (c) at (0,0) {\vdot};
        \draw (-1.5,-2.5)--(-1.7,-2.5)--(-1.7,0.5)--(-1.5,0.5);
        \draw (1.5,-2.5)--(1.7,-2.5)--(1.7,0.5)--(1.5,0.5);
        \draw (b.center)--(a1.center)--(c.center);
    \end{tikzpicture}
    };
    \node (11) at (-1.75,10.5){
    \begin{tikzpicture}[scale=0.3]
        \node (a1) at (-1,-2) {\vdot};
        \node (a2) at (1,-2) {\vdot};
        \node (b) at (0,-1) {\vdot};
        \node (c) at (0,0) {\vdot};
        \draw (-1.5,-2.5)--(-1.7,-2.5)--(-1.7,0.5)--(-1.5,0.5);
        \draw (1.5,-2.5)--(1.7,-2.5)--(1.7,0.5)--(1.5,0.5);
        \draw (b.center)--(a2.center);
    \end{tikzpicture}
    };
    \node (12) at (1.75,10.5){
    \begin{tikzpicture}[scale=0.3]
        \node (a1) at (-1,-2) {\vdot};
        \node (a2) at (1,-2) {\vdot};
        \node (b) at (0,-1) {\vdot};
        \node (c) at (0,0) {\vdot};
        \draw (-1.5,-2.5)--(-1.7,-2.5)--(-1.7,0.5)--(-1.5,0.5);
        \draw (1.5,-2.5)--(1.7,-2.5)--(1.7,0.5)--(1.5,0.5);
        \draw (b.center)--(a1.center);
    \end{tikzpicture}
    };
    \node (13) at (0,12){
    \begin{tikzpicture}[scale=0.3]
        \node (a1) at (-1,-2) {\vdot};
        \node (a2) at (1,-2) {\vdot};
        \node (b) at (0,-1) {\vdot};
        \node (c) at (0,0) {\vdot};
        \draw (-1.5,-2.5)--(-1.7,-2.5)--(-1.7,0.5)--(-1.5,0.5);
        \draw (1.5,-2.5)--(1.7,-2.5)--(1.7,0.5)--(1.5,0.5);
    \end{tikzpicture}
    };
    \draw (4) to [bend left=22] (13);
    \filldraw[white] (-0.88,6.77) circle (1.7pt);
    \filldraw[white] (-0.79,9.68) circle (1.7pt);
    \filldraw[white] (-0.98,8.16) circle (1.7pt);
    \draw (1)--(2)--(4)--(3)--(1);
    \draw (5)--(6)--(9)--(7)--(5);
    \draw (6)--(8)--(11)--(9)--(12)--(13)--(11);
    \draw (7)--(10)--(12);
    \filldraw[white] (1.02,3.21) circle (1.9pt);  
    \draw (1) to [bend right=35] (5);
    \draw (2)--(8);
    \draw (3)--(10);
\end{tikzpicture}}
\end{tabular}
\caption{The lattice $\Tr(Y)$ and its image under $\gamma_Y$.}
\label{fig:yop}
\end{figure}
\end{example}

\appendix
\section{A bijection with coalescent histories}\label{ap:coalhis}

The recurrence in Theorem \ref{thm:recur} was motivated by example generation. We first computed the number of transfer systems on full binary trees of depths $1$, $2$, and $3$ using elementary counting techniques, and then we entered these values (4, 169, 1020100) into the Online Encyclopedia of Integer Sequences (OEIS). We found a match with A306402 \cite{oeisbintree}. This yielded a sequence from computational biology for the number of ``valid coalescent histories on full binary trees'' as computed in \cite{rosenberg_counting_2007}. While there isn't a closed form for this sequence, Rosenberg produced a recurrence relation for a class of coalescent histories \cite{rosenberg_counting_2007}, one example of which gives rise to this sequence. This motivated us to find a bijection between coalescent histories and transfer systems on binary trees, as well as to find a similar recurrence for transfer systems on more general arborescences (see Section~\ref{sec:recur}). In this appendix we introduce coalescent histories on bifurcated trees and show there is a bijection with transfer systems on the subtree given by removing the leaves. \smallskip

Coalescent histories are structures defined on a pair of rooted trees, a \emph{gene tree} \(G\) and a \emph{species tree} \(S\). These structures model how a gene could have evolved across a group of species. The leaves represent the group studied, whereas internal nodes represent ancestral populations. Mathematically, we require the leaf set of \(G\) to be a subset of the leaf set of \(S\), and both trees should be \emph{bifurcated}, meaning each internal node (as well as the root) has exactly $2$ children. We can view the vertex set of the tree as a poset where the edges give us relations so that moving back in time in the gene/species tree corresponds to decreasing in the poset. Thus the gene and species trees are examples of arborescences (see Remark~\ref{rem:treetoarb}).

In \cite{rosenberg_counting_2007}, following the definition given in \cite{DegnanSalter2005}, a \emph{valid coalescent history} is a function $h$ from the set of internal nodes (not including the root) of the gene tree \(G\) to the set of non-leaf vertices (including the root) of the species tree \(S\) such that 
\begin{enumerate}
    \item $h$ is order-preserving; and
    \item for all leaves $\ell_1, \ell_2$ in $G$, we have that $h(\ell_1 \meet \ell_2) \leq \ell_1 \meet \ell_2$.
\end{enumerate} 
\begin{remark}
In \cite{DegnanSalter2005} and \cite{rosenberg_counting_2007} the authors define coalescent histories as functions whose codomain is the set of internal edges of the tree $S$, where they add an additional edge pointing to the root (this symbolizes ancestors that existed before the root of the species tree). Recall we view edges as directed away from the root, and so each edge corresponds uniquely with its target. Thus we can instead regard the codomain of these functions to be the set of non-leaf vertices of $S$. We also note the authors do not use the language of posets, but instead refer to how ``ancient'' or ``recent'' lineages are (which coincides with relations in the poset) and when two lineages ``coalesce'' (which coincides with meets). We have translated their definitions to match our poset language given here.
\end{remark}

In this paper we are only interested in the case where the gene and species tree are equal. Then we can reinterpret the definition as follows:
\begin{defn}
\label{def:coalescent}
Let \(T\) be an arborescence whose Hasse diagram is a bifurcated tree. Write \(\widehat{T}\) for the subposet of \(T\) consisting of the internal nodes along with the root. A (valid) \textbf{coalescent history} with matching gene and species tree given by \(T\) is a function \(h\colon \widehat{T} \to \widehat{T}\) such that
    \begin{enumerate}
        \item \(h\) is order-preserving, i.e. for all \(x,y\in\widehat{T}\), if \(x\leq y\) then \(h(x)\leq h(y)\); and
        \item for all \(x\in \widehat{T}\), \(h(x) \leq x\).
    \end{enumerate}
\end{defn}
Note we have now included the root in the domain, but we will always have $h(\bot)=\bot$ by condition (2). So this has no effect other than simplifying the definition/notation. We leave it as an exercise for the reader to verify that condition (2) given above is equivalent to the condition $h(\ell_1 \meet \ell_2) \leq \ell_1 \meet \ell_2$ for all leaves $\ell_1, \ell_2\in G$ in the case that $G=S$. The key observation is that each internal node is the meet of some unique list of leaves in our bifurcated tree.

Recall from Section~\ref{sec:bijmono} that for an arborescence $\A$ we defined
\[
M_\A=\{f\colon \A \to \N \mid f \leq \rho \text{ and } f \text{ is order-preserving}\}.
\]
We now state and prove our main theorem from this section:
\begin{thm}\label{thm:bijcoalhis}
Let $T$ be an arborescence whose Hasse diagram is a bifurcated tree. Write $C_T$ for the set of valid coalescent histories whose gene and species trees are given by $T$. Then we have a bijection from $C_T$ onto $M_{\widehat{T}}$ and so $\# \Tr(\widehat{T})=\# C_T$.
\end{thm}
\begin{proof}
We define a map $\Phi\colon C_T\to M_{\widehat{T}}$. Let $h\in C_T$ and define $\Phi(h)=\rho \circ h\colon \widehat{T}\to \N$ where $\rho$ is usual rank function that measures distance from the root. It is straightforward to verify $\Phi(h)\in M_{\widehat{T}}$. For the inverse define $\Psi\colon M_{\widehat{T}}\to C_T$ as follows. Let $f\in M_{\widehat{T}}$ and $x\in \widehat{T}$. We have $f(x)\leq \rho(x)$, so there exists a unique ancestor $a$ of $x$ such that $\rho(a)=f(x)$. Let $\Psi(f)(x)=a$. One can verify $\Phi\circ \Psi = \id$ and $\Psi \circ \Phi = \id$. Finally by Theorem~\ref{thm:arbij} we also have a bijection $M_{\widehat{T}}\to \Tr(\widehat{T})$ so we conclude $\#\Tr(\widehat{T})=\#C_T$.  
\end{proof}

\bibliographystyle{amsalpha}
\bibliography{trees}

@article {selfdual,
    AUTHOR = {Franchere, Evan E. and Ormsby, Kyle and Osorno, Ang\'{e}lica
              M. and Qin, Weihang and Waugh, Riley},
     TITLE = {Self-duality of the lattice of transfer systems via weak
              factorization systems},
   JOURNAL = {Homology Homotopy Appl.},
  FJOURNAL = {Homology, Homotopy and Applications},
    VOLUME = {24},
      YEAR = {2022},
    NUMBER = {2},
     PAGES = {115--134},
      ISSN = {1532-0073,1532-0081},
   MRCLASS = {55P91 (18A32)},
MRREVIEWER = {Steven\ R.\ Costenoble},
       DOI = {10.4310/hha.2022.v24.n2.a6},
       URL = {https://doi.org/10.4310/hha.2022.v24.n2.a6},
}

@article{ninfty,
      title={ninfty: A software package for homotopical combinatorics}, 
      author={Scott Balchin},
      year={2025},
      JOURNAL = {arXiv preprint  	arXiv:2504.01003},
}

@article {BBR1,
    AUTHOR = {Balchin, Scott and Barnes, David and Roitzheim, Constanze},
     TITLE = {{$N_\infty$}-operads and associahedra},
   JOURNAL = {Pacific J. Math.},
  FJOURNAL = {Pacific Journal of Mathematics},
    VOLUME = {315},
      YEAR = {2021},
    NUMBER = {2},
     PAGES = {285--304},
      ISSN = {0030-8730,1945-5844},
   MRCLASS = {18M80 (06A07 52B20 55N91 55P91)},
       DOI = {10.2140/pjm.2021.315.285},
       URL = {https://doi.org/10.2140/pjm.2021.315.285},
}

@article {BMO,
    AUTHOR = {Balchin, Scott and MacBrough, Ethan and Ormsby, Kyle},
     TITLE = {The combinatorics of {$N_\infty$} operads for {$C_{qp^n}$} and
              {$D_{p^n}$}},
   JOURNAL = {Glasg. Math. J.},
  FJOURNAL = {Glasgow Mathematical Journal},
    VOLUME = {67},
      YEAR = {2025},
    NUMBER = {1},
     PAGES = {50--66},
      ISSN = {0017-0895,1469-509X},
   MRCLASS = {55P48 (18M80)},
MRREVIEWER = {Steffen\ Sagave},
       DOI = {10.1017/S0017089524000211},
       URL = {https://doi.org/10.1017/S0017089524000211},
}

@article {BMO2,
    AUTHOR = {Balchin, Scott and MacBrough, Ethan and Ormsby, Kyle},
     TITLE = {Lifting {$N_{\infty}$} operads from conjugacy data},
   JOURNAL = {Tunis. J. Math.},
  FJOURNAL = {Tunisian Journal of Mathematics},
    VOLUME = {5},
      YEAR = {2023},
    NUMBER = {3},
     PAGES = {479--504},
      ISSN = {2576-7658,2576-7666},
   MRCLASS = {55P91 (06B05 20D30 20J99)},
MRREVIEWER = {Andr\'{e}\ G.\ Henriques},
       DOI = {10.2140/tunis.2023.5.479},
       URL = {https://doi.org/10.2140/tunis.2023.5.479},
}

@article {modelstr,
    AUTHOR = {Balchin, Scott and Ormsby, Kyle and Osorno, Ang\'{e}lica M.
              and Roitzheim, Constanze},
     TITLE = {Model structures on finite total orders},
   JOURNAL = {Math. Z.},
  FJOURNAL = {Mathematische Zeitschrift},
    VOLUME = {304},
      YEAR = {2023},
    NUMBER = {3},
     PAGES = {Paper No. 40, 35},
      ISSN = {0025-5874,1432-1823},
   MRCLASS = {55U35 (18N40 18N70)},
MRREVIEWER = {Philippe\ Gaucher},
       DOI = {10.1007/s00209-023-03287-6},
       URL = {https://doi.org/10.1007/s00209-023-03287-6},
}

@article {steiner,
    AUTHOR = {Rubin, Jonathan},
     TITLE = {Detecting {S}teiner and linear isometries operads},
   JOURNAL = {Glasg. Math. J.},
  FJOURNAL = {Glasgow Mathematical Journal},
    VOLUME = {63},
      YEAR = {2021},
    NUMBER = {2},
     PAGES = {307--342},
      ISSN = {0017-0895,1469-509X},
   MRCLASS = {55P91 (55P48)},
MRREVIEWER = {Ben\ C.\ Walter},
       DOI = {10.1017/S001708952000021X},
       URL = {https://doi.org/10.1017/S001708952000021X},
}

@book{stanley1,
  author    = {Stanley, Richard P.},
  title     = {Enumerative Combinatorics, Volume 1},
  series    = {Cambridge Studies in Advanced Mathematics},
  volume    = {49},
  edition   = {2nd},
  year      = {1997},
  publisher = {Cambridge University Press},
  address   = {Cambridge},
  isbn      = {978-0521553094}
}

@article {rubinoperads,
    AUTHOR = {Rubin, Jonathan},
     TITLE = {Combinatorial {$N_\infty$} operads},
   JOURNAL = {Algebr. Geom. Topol.},
  FJOURNAL = {Algebraic \& Geometric Topology},
    VOLUME = {21},
      YEAR = {2021},
    NUMBER = {7},
     PAGES = {3513--3568},
      ISSN = {1472-2747,1472-2739},
   MRCLASS = {55P48 (55P91)},
MRREVIEWER = {Jonathan\ Beardsley},
       DOI = {10.2140/agt.2021.21.3513},
       URL = {https://doi.org/10.2140/agt.2021.21.3513},
}

@article{modelstr2,
    title={Characterizing model structures on finite posets},
    author={Mazur, Kristen and Osorno, Ang\'{e}lica M. and Roitzheim, Constanze and Santhanam, Rekha and Van Niel, Danika and Zapata Castro, Valentina},
    year={2025},
    JOURNAL = {arXiv preprint  	arXiv:2511.06151}
}

@article{bousfield1,
    title={Left and right Bousfield localization on lattices},
    author={Carnero Bravo, Andr\'{e}s and Goyal, Shuchita and Mart\'{\i}nez Alberga, Sof\'{\i}a and Ng, Cherry and Roitzheim, Constanze and Tolosa, Daniel},
    year={2025},
    JOURNAL = {arXiv preprint  	 	 	arXiv:2511.07952}
}

@article{bousfield2,
    title={Bousfield Localizations on the Nonmodular Lattice {$N_5$}},
    author={Mart\'{\i}nez Alberga, Sof\'{\i}a and Roitzheim, Constanze},
    year={2026},
    JOURNAL = {arXiv preprint  	 	arXiv:2605.12744}
}

@article{nonabelian,
    title={Characterizing Transfer Systems for Non-Abelian Groups},
    author={Klanderman, Sarah and Lewis, Chloe and Monson, Harlea and Shibata, Koki and Van Niel, Danika},
    year={2025},
    JOURNAL = {arXiv preprint arXiv:2511.13439}
}

@article {ranktwo,
    AUTHOR = {Bao, Linus and Hazel, Christy and Karkos, Tia and Kessler,
              Alice and Nicolas, Austin and Ormsby, Kyle and Park, Jeremie
              and Schleff, Cait and Tilton, Scotty},
     TITLE = {Transfer systems for rank two elementary abelian groups:
              characteristic functions and matchstick games},
   JOURNAL = {Tunis. J. Math.},
  FJOURNAL = {Tunisian Journal of Mathematics},
    VOLUME = {7},
      YEAR = {2025},
    NUMBER = {1},
     PAGES = {167--191},
      ISSN = {2576-7658,2576-7666},
   MRCLASS = {06B05 (20K10 55P91)},
       DOI = {10.2140/tunis.2025.7.167},
       URL = {https://doi.org/10.2140/tunis.2025.7.167},
}

@article {BH1,
    AUTHOR = {Blumberg, Andrew J. and Hill, Michael A.},
     TITLE = {Operadic multiplications in equivariant spectra, norms, and
              transfers},
   JOURNAL = {Adv. Math.},
  FJOURNAL = {Advances in Mathematics},
    VOLUME = {285},
      YEAR = {2015},
     PAGES = {658--708},
      ISSN = {0001-8708,1090-2082},
   MRCLASS = {55P91 (18D50 55P43 55P48)},
MRREVIEWER = {Markus\ Szymik},
       DOI = {10.1016/j.aim.2015.07.013},
       URL = {https://doi.org/10.1016/j.aim.2015.07.013},
}

@article {HMOO,
    AUTHOR = {Hafeez, Usman and Marcus, Peter and Ormsby, Kyle and Osorno,
              Ang\'{e}lica M.},
     TITLE = {Saturated and linear isometric transfer systems for cyclic
              groups of order {$p^mq^n$}},
   JOURNAL = {Topology Appl.},
  FJOURNAL = {Topology and its Applications},
    VOLUME = {317},
      YEAR = {2022},
     PAGES = {Paper No. 108162, 20},
      ISSN = {0166-8641,1879-3207},
   MRCLASS = {55P91 (55P43 55P48)},
MRREVIEWER = {David\ Barnes},
       DOI = {10.1016/j.topol.2022.108162},
       URL = {https://doi.org/10.1016/j.topol.2022.108162},
}

@article{rosenberg_counting_2007,
  author    = {Noah A. Rosenberg},
  title     = {Counting Coalescent Histories},
  journal   = {Journal of Computational Biology},
  volume    = {14},
  number    = {3},
  pages     = {360--377},
  year      = {2007},
  doi       = {10.1089/cmb.2006.0109},
  pmid      = {17563317},
  url       = {https://doi.org}
}

@misc{oeisbintree,
  author = {{OEIS Foundation Inc.}},
  title = {Entry {A306402} in the {On-Line Encyclopedia of Integer Sequences}},
  howpublished = {Published electronically at \url{https://oeis.org/A306402}},
  year = {2026}
}

@article{DegnanSalter2005,
  author  = {Degnan, James H. and Salter, Laura A.},
  title   = {Gene tree distributions under the coalescent process},
  journal = {Evolution},
  volume  = {59},
  number  = {1},
  pages   = {24--37},
  year    = {2005},
  doi     = {10.1111/j.0014-3820.2005.tb00891.x}
}

@book {algtopbook,
    AUTHOR = {May, J. P. and Ponto, K.},
     TITLE = {More concise algebraic topology},
    SERIES = {Chicago Lectures in Mathematics},
      NOTE = {Localization, completion, and model categories},
 PUBLISHER = {University of Chicago Press, Chicago, IL},
      YEAR = {2012},
     PAGES = {xxviii+514},
      ISBN = {978-0-226-51178-8; 0-226-51178-2},
   MRCLASS = {55-02 (16T05 18G55 55P60)},
  MRNUMBER = {2884233},
MRREVIEWER = {Ismar\ Voli\'{c}},
}

\end{document}